\documentclass{amsart}
\usepackage[utf8]{inputenc}
\usepackage{latexsym,amsxtra,amscd,ifthen}
\usepackage{amsfonts}
\usepackage{color,graphicx}
\usepackage{verbatim}
\usepackage{amsmath}
\usepackage{amsthm}
\usepackage{amssymb}
\usepackage{url}
\usepackage{ulem}

\usepackage{tikz}
\usetikzlibrary{arrows,shapes,chains}

\theoremstyle{plain}

\newtheorem{theorem}{Theorem}
\newtheorem{lemma}[theorem]{Lemma}
\newtheorem{proposition}[theorem]{Proposition}
\newtheorem{corollary}[theorem]{Corollary}

\numberwithin{theorem}{section}
\numberwithin{equation}{theorem}

\theoremstyle{definition}
\newtheorem{definition}[theorem]{Definition}
\newtheorem{example}[theorem]{Example}

\newtheorem{question}[theorem]{Question}
\newtheorem*{question*}{Question}

\newcommand{\Z}{\mathbb{Z}}

\newcommand{\kk}{\Bbbk}

\DeclareMathOperator{\Frac}{Frac}

\DeclareMathOperator{\Mod}{Mod}

\DeclareMathOperator{\Aut}{Aut}
\DeclareMathOperator{\gr}{gr}

\DeclareMathOperator{\moc}{c}

\DeclareMathOperator{\CompS}{CompS}
\DeclareMathOperator{\CompM}{CompM}
\begin{document}

\title{Relative Dixmier property for Poisson algebras}

\author{H.-D. Huang, Z. Nazemian, 
X. Tang, X.-T. Wang, Y.-H. Wang, and J.J. Zhang}

\address{Hongdi Huang: Department of Mathematics, Shanghai University,
Shanghai, 20244, China}

\email{hdhuang@shu.edu.cn}

\address{Zahra Nazemian: University of Graz, Heinrichstrasse 36, 8010 Graz, Austria}

\email{zahra.nazemian@uni-graz.at}

\address{Xin Tang: Department of Mathematics \& Computer Science, 
Fayetteville State University, Fayetteville, NC 28301,
USA}

\email{xtang@uncfsu.edu}

\address{Xingting Wang: 
Department of Mathematics, Louisiana State University, 
Baton Rouge, Louisiana 70803, USA} 

\email{xingtingwang@math.lsu.edu}

\address{Yanhua Wang: School of Mathematics, Shanghai Key Laboratory of Financial Information Technology, Shanghai University of Finance and Economics, Shanghai 200433, China}

\email{yhw@mail.shufe.edu.cn}

\address{James J. Zhang: Department of Mathematics, Box 354350,
University of Washington, Seattle, Washington 98195, USA}

\email{zhang@math.washington.edu}

\begin{abstract}
Dixmier property concerns the bijectivity of endomorphisms for algebras. We introduce a relative Dixmier property, which is a generalization of the Dixmier property. This new concept has applications in proving that several classes of Poisson algebras possess the Dixmier property, as well as in other topics such as the cancellation problem and the non-existence of Hopf coactions. 
\end{abstract}




\maketitle


\section*{Introduction}
\label{xxsec0}

Poisson algebra has been an active research area in recent years. Various topics related to Poisson algebras have been studied. In particular, there have been significant developments on twisted Poincar{\' e} duality and the modular derivation, Poisson Dixmier-Moeglin equivalence, Poisson enveloping algebras, Poisson spectrum, and Poisson valuations, see papers \cite{BLSM, Goo1, GL, HTWZ1, HTWZ2, JO, LS, LuWW1, LuWW2, LvWZ} and the references therein. 

Let $\Bbbk$ be a base field and assume that everything is over $\Bbbk$. A Poisson $\Bbbk$-algebra $A$ is said to satisfy the {\it Dixmier property} (or simply $A$ is {\it Dixmier}) if every injective Poisson $\Bbbk$-algebra morphism $f\colon A\longrightarrow A$ is bijective. This property is related to the well-known Dixmier conjecture \cite{Dix}, which states that every endomorphism of the $n$-th Weyl algebra over a field of characteristic zero is bijective. The Dixmier conjecture is proved to be stably equivalent to the infamous Jacobian conjecture  \cite{BCW, BK, Tsu1, Tsu2}. By \cite{AvdE}, these two conjectures are further equivalent to 
the Poisson conjecture, which states that every Poisson endomorphism of the Poisson Weyl algebra $W_n$ [see Example \ref{xxexa1.1}] over a field of characteristic zero is bijective. Basically, the Poisson conjecture asserts that $W_n$ satisfies the Dixmier property. 

Note that the Dixmier conjecture still remains wide open. Though it seems rare for a given algebra to satisfy the Dixmier property, this property is important to study. As a matter of fact, the Dixmier property has been extensively studied for several classes of associative algebras of a `quantum nature', see the papers \cite{AW, KL1, KL2, Ric, Tan1, Tan2, Tan3}. Motivated by the close relationship between `quantum algebras' and their semiclassical limits \cite{Goo2}, it is natural to study the Dixmier property for Poisson algebras. In this paper, we focus on the Dixmier property in the Poisson setting. From now on, we will assume that the base field $\Bbbk$ is of characteristic zero.

There have been examples of Poisson algebras which satisfy the Dixmier property \cite[Theorem 3.7]{CO}. In \cite[Theorem 0.9]{HTWZ1}, it is also proved that the following Poisson algebra $P_{\Omega-\xi}$ is Dixmier.

\begin{example}
\label{xxexa0.1}
Let $\Omega$ be a homogeneous element in 
$A:=\Bbbk[x,y,z]$ with Adams grading defined by $\deg x=\deg y=\deg z=1$. We define a Poisson structure on $A$ by 
$\{x,y\}=\frac{\partial \Omega}{\partial z}  =: \Omega_{z}$,
$\{y,z\}=\frac{\partial \Omega}{\partial x}  =: \Omega_{x}$,
and
$\{z,x\}=\frac{\partial \Omega}{\partial y}  =: \Omega_{y}$.
Equivalently, the Poisson bracket on $A$ is 
defined by
\begin{equation}\label{F0.1.1}
\tag{F0.1.1}
\{f,g\}=\det \begin{pmatrix} 
f_{x} & f_{y} & f_{z}\\
g_{x} & g_{y} & g_{z}\\
\Omega_{x} & \Omega_{y} & \Omega_{z}
\end{pmatrix}\quad 
\text{for all $f,g\in \Bbbk[x, y, z]$}.
\end{equation}
This Poisson algebra is denoted by 
$A_{\Omega}$ and $\Omega$ is called the 
{\it potential} of $A_{\Omega}$. 
It is easy to check that $\Omega$ is a
Poisson central element. For every 
$\xi\in \Bbbk$, we define a Poisson factor
algebra $P_{\Omega-\xi}:=
A_{\Omega}/(\Omega-\xi)$. We say a homogeneous 
element $\Omega\in A$  
{\it has an isolated singularity at the origin} 
{\rm{(}}or simply {\it has an isolated 
singularity}{\rm{)}} if $A_{sing}:=
A/(\Omega_x,\Omega_y,\Omega_z)$ is finite 
dimensional over $\Bbbk$. By commutative 
algebra, $\Omega$ has an  isolated singularity 
if and only if $\{\Omega_x, \Omega_y, \Omega_z\}$
is a regular sequence in the polynomial ring
$A$. In this case, we say $\Omega$ is an {\it i.s. potential} where i.s. stands for isolated singularity. Note that a generic homogeneous element of degree $\geq 2$ has an isolated singularity. By \cite[Theorem 0.9]{HTWZ1}, if $\Omega$ is an i.s. potential of degree $\geq 5$, then for every $\xi\in\Bbbk$, $P_{\Omega-\xi}$ is Dixmier.
\end{example}

The proof of $P_{\Omega-\xi}$ being Dixmier is quite involved and uses some deep analysis 
of Poisson valuations. On the other hand, one is also wondering if the tensor products of $P_{\Omega-\xi}$ 
are also  Dixmier. In this paper, we will explore additional ideas that may help us better understand the Dixmier property. The following is one of the main results. For any Poisson $\Bbbk$-algebra $A$, we let $\Aut_{Poi.alg}(A)$
denote the group of all Poisson $\Bbbk$-algebra automorphisms of $A$.

\begin{theorem}
\label{xxthm0.2}
Let $\Omega_1,\cdots,\Omega_d$ be a set of i.s. potentials of distinct Adams degrees with 
$\deg \Omega_i\geq 5$ for all $i$. Let $\xi_1,\cdots,\xi_d$ be a set of scalars in 
$\Bbbk$. Let $P$ be the Poisson algebra 
$P_{\Omega_1-\xi_1}\otimes \cdots \otimes 
P_{\Omega_d-\xi_d}$ and let $Q$ be the fractional Poisson field of $P$. The following statements are true.
\begin{enumerate}
\item[(1)]
Both $P$ and $Q$ are Dixmier.
\item[(2)]
The groups $\Aut_{Poi.alg}(P)$ and $\Aut_{Poi.alg}(Q)$ are finite.
\item[(3)]
The Poisson algebra $P$ is simple if and only if $\xi_{i}\neq 0$ 
for all $i=1,\cdots,d$.
\end{enumerate}
\end{theorem}

However, we do not know how to handle the case when two of $\deg \Omega_i$ are the same; so 
we leave it as a question.
\begin{question}
\label{xxque0.3}
Let $\Omega_1,\cdots,\Omega_d$ be a set of i.s.
potentials of Adams degree $\geq 5$. Let 
$\xi_1,\cdots,\xi_d$ be a set of scalars in 
$\Bbbk$. Let $P$ be the Poisson algebra 
$P_{\Omega_1-\xi_1}\otimes \cdots \otimes
P_{\Omega_d-\xi_d}$.
\begin{enumerate}
\item[(1)]
Is $P$ Dixmier?
\item[(2)]
Is $\Aut_{Poi.alg}(P)$ finite?
\end{enumerate}
\end{question}

A key step in proving Theorem \ref{xxthm0.2} involves the idea of so-called {\it relative Dixmier property}, which is 
motivated by the notion of relative cancellation as introduced in \cite{HNWZ1}. We have the following definition. 

 \begin{definition}
\label{xxdef0.4}
Let ${\mathcal A}$ and ${\mathcal R}$ be 
two families of Poisson $\Bbbk$-algebras. 
\begin{enumerate}
\item[(1)]
We say ${\mathcal A}$ is {\it 
${\mathcal R}$-Dixmier} if for every two 
Poisson algebras $A_1, A_2$ in ${\mathcal A}$ 
and a Poisson algebra $R$ in ${\mathcal R}$, 
for every injective Poisson $\Bbbk$-algebra morphism 
$f\colon A_1\longrightarrow A_2\otimes R$, we have 
${\rm{im}}(f)=A_2\otimes \Bbbk$ inside 
$A_2\otimes R$.
\item[(2)]
Let $A$ and $R$ be Poisson algebras. We 
say $A$ is {\it $R$-Dixmier} if 
${\mathcal A}$ is ${\mathcal R}$-Dixmier 
for ${\mathcal A}=\{A\}$ and ${\mathcal R}
=\{R\}$. 
\end{enumerate}
\end{definition}

It is clear that, if a Poisson algebra $A$ is $R$-Dixmier for a Poisson algebra $R$, then $A$ is 
Dixmier. However, there are examples in which a Poisson algebra $A$ is Dixmier, but not $R$-Dixmier 
for some chosen Poisson domain $R$, see Example \ref{xxexa1.6}. 

\begin{example}
\label{xxexa0.5}
For any $n\geq 2$, let $\Lambda=(\lambda_{ij})=(\lambda m_{ij})$ where $M=(m_{ij}) \in M_{n}(\Z)$ is a skew-symmetric $n\times n$ integer matrix and $\lambda \in \Bbbk^{\times}$. Let $L_{\Lambda}(\Bbbk)$ or simply $L_{\Lambda}$ denote the Poisson torus $\Bbbk[x_{1}^{\pm 1}, \cdots, x_{n}^{\pm 1}]$ 
with the following Poisson bracket:
\[
\{x_{i},x_{j}\}=\lambda_{ij} x_{i}x_{j}
\]
for $1\leq i\leq j\leq n$. Suppose that $L_{\Lambda}(\Bbbk)$ is Poisson simple. By Theorem \ref{xxthm2.6}, $L_{\Lambda}(\Bbbk)$ 
satisfies the Dixmier property. Furthermore, by using Theorem \ref{xxthm2.4}, $L_{\Lambda}(\Bbbk)$ is also $R$-Dixmier whenever $R$ is a connected graded Poisson domain (with respect to an Adams grading as a commutative algebra).
\end{example}

Another main result of this paper is to show that the Poisson algebras in Theorem \ref{xxthm0.2} are $\mathcal{R}$-Dixmier for a certain class $\mathcal{R}  $ of Poisson algebras. For any integer $w\in \Z$, let ${\mathcal G}_{w}$ be the class of Adams connected graded Poisson domains $R = \bigoplus_{i \geq 0} R_i $ generated in degree 1 such that $\{R_1, R_1\} \subseteq R_{w+2}$.

\begin{theorem}
\label{xxthm0.6}
Let $\Omega_1,\cdots,\Omega_d$ be a set of i.s. potentials of distinct Adams degrees with $\deg \Omega_i\geq 5$ for each 
$i=1, \cdots, d$. Let $\xi_1,\cdots,\xi_d$ be a set of scalars in $\Bbbk$. Let $P$ be the Poisson algebra $P_{\Omega_1-\xi_1}\otimes \cdots \otimes P_{\Omega_d-\xi_d}$. Let $L_{\Lambda}(\Bbbk)$ be a simple Poisson torus as defined in Example \ref{xxexa0.5}. Let $A$ be either $P$ or $L_{\Lambda}(\Bbbk) \otimes P$. Suppose $w\leq \min_{i}\{\deg \Omega_i\} -1$.
Then $A$ is ${\mathcal G}_{w}$-Dixmier.
\end{theorem}

As an immediate consequence, we have the following corollary (see Definition \ref{xxdef2.16} and the references \cite{HNWZ1, Mon} for the relevant concepts).

\begin{corollary}
\label{xxcor0.7}
Let $A$ and $w$ be as in Theorem \ref{xxthm0.6}.
\begin{enumerate}
\item[(1)]
$A$ is ${\mathcal G}_w$-cancellative.
\item[(2)]
If $H$ is a Poisson Hopf algebra that is in ${\mathcal G}_w$ as a Poisson algebra, then every 
$H$-coaction on $A$ is trivial. 
\end{enumerate}
\end{corollary}

In general, it is not easy to prove that a Poisson algebra has the Dixmier property, but it does not necessarily
mean that Poisson algebras with the Dixmier property are rare. At the end of the paper, we will consider 
the following question:

\begin{question}
\label{xxque0.8}
Suppose a Poisson algebra $B$ satisfies the 
Dixmier property and $A$ is a Poisson 
subalgebra of $B$. Under what conditions does $A$ 
satisfy the Dixmier property?
\end{question}

We are able to answer this question in a special case:

\begin{theorem}
\label{xxthm0.9}
Suppose that a Poisson domain $B$ is integrally closed as a commutative algebra, and Dixmier, and 
$A$ is a cofinite-dimensional Poisson subalgebra of $B$ such that $B\subseteq Q(A)$. Then 
$A$ also satisfies the Dixmier property.

\end{theorem}

A relative version of the above theorem is Lemma~\ref{xxlem4.8}, whose proof is based on the introduction of the concepts of complementary spaces and complement modules in Section~4. We show that these concepts are invariants that help us better understand the Poisson automorphism groups of a large class of algebras. In particular, we identify a class of algebras whose Poisson algebra automorphism groups are trivial.

\section{Some examples of Poisson algebras}
\label{xxsec1}

In this section, we recall some important classes of Poisson algebras relevant to the studies in the next few sections. The first one is the Poisson Weyl algebra, which is the object appearing in the Poisson conjecture.

\begin{example}
\label{xxexa1.1}
Let $W_1$ be the Poisson polynomial algebra $\Bbbk[x, y]$ generated by $x$ and $y$ with the Poisson 
bracket determined by $\{x,y\}=1$. This is called the {\it first Poisson Weyl algebra}. For $n\geq 2$, the 
$n$-th Poisson Weyl algebra is defined to be $W_n:=(W_1)^{\otimes n}$. It is well known that $W_{n}$ is Poisson simple for any $n\geq 1$. The Poisson conjecture asserts that every {nonzero} Poisson algebra endomorphism 
of $W_n$ is bijective. 
\end{example}

\begin{example}
\label{xxexa1.2}
Let $W_1$ be the first Poisson Weyl algebra generated by $x$ and $y$. The localization of $W_1$ by inverting $x$ is a Poisson algebra, and it is denoted by $W_1[x^{-1}]$. We claim that $A:=W_1[x^{-1}]$ does not have the Dixmier property. To see that, we let $X=x^2$ and $Y=\frac{1}{2} x^{-1}y$. It is easy to check that $\{X,Y\}=1$. Define $f: A\longrightarrow A$ by $f(x)=X$ and $f(y)=Y$. Then $f$ extends to an injective Poisson algebra endomorphism of $A$ with $x\not\in f(A)$. Therefore, $f$ is not a bijective map as desired.
\end{example}

\begin{example}
\label{xxexa1.3}
Let $n$ be an integer $\geq 2$ and $\Lambda=(\lambda_{ij})\in M_{n}(\Bbbk)$ be an $n\times n$ skew-symmetric matrix. The skew Poisson polynomial algebra associated to $\Lambda$ is defined to be the polynomial algebra $\Bbbk[x_1,\cdots,x_n]$ with its Poisson bracket determined by $$\{x_i,x_j\}=\lambda_{ij} x_i x_j, \quad \forall\; 1\leq i, j\leq n.$$ This Poisson algebra is denoted by $S_{\Lambda}(\Bbbk)$. 
\end{example}

\begin{example}
\label{xxexa1.4}
Let $S_{\Lambda}(\Bbbk)$ be the Poisson Polynomial algebra defined in Example \ref{xxexa1.3}. Its localization
by inverting $x_i$ for all $i$ is called the Poisson Laurent polynomial algebra (or Poisson torus) and is denoted by
$L_{\Lambda}(\Bbbk)$. 
\end{example}

Note that $L_{\Lambda}(\Bbbk)$ is not necessarily Dixmier as shown by the example below. 
\begin{example}
\label{xxexa1.5} 
Let $n = 3$ and $\lambda_{ij}=1$ for $1\leq i<j\leq 3$ and define $ \phi$ as an algebra endomorphism of 
$L_{\Lambda}(\Bbbk)$ 
such that 
\[
\phi(x_1) = x_1^3 x_2^{-2} x_3^2, \qquad
\phi(x_2) = x_1 x_3, \qquad
\phi(x_3) = x_3.
\]
Then $x_2 $ is not in the image of $\phi$. Now let us check if $\phi$ is a Poisson morphism. It is not difficult to see that $$\phi (\{x_1, x_2\}) = \phi ( x_1 x_2 ) = x_1^4 x_2^{-2} x_3^3 = \{\phi (x_1), \phi(x_2)\}$$ and also $$\phi (\{x_1, x_3\}) = \phi ( x_1 x_3 ) = x_1^3 x_2^{-2} x_3^3 = \{\phi (x_1), \phi(x_3)\}$$ and also clearly $\phi (\{x_2, x_3\}) = x_1 x_3^2 = \{\phi (x_2), \phi(x_3)\}$. So $\phi$ is a morphism of Poisson algebra. Note that $\phi ( x_1 ^{n_1} x_2 ^{n_2} x_3^{n_3}) =  x_1 ^{3n_1+ n_2} x_2 ^{-2n_2} x_3^{2n_1 + n_2 + n_3}$. Thus $ \phi $ is injective but not surjective. This shows that for $n = 3$, $L_{\Lambda}(\Bbbk)$ does not satisfy the Dixmier property. Note that in this case, $L_{\Lambda}(\Bbbk)$ is not Poisson simple; indeed, $x_{1}x_{2}^{-1}x_{3}$ is in the Poisson center.  

\end{example}

Next, we give an example which is Dixmier, but not $R$-Dixmier for some Poisson domain $R$.
\begin{example}
\label{xxexa1.6}
Let $n=2$. As we will see in Example \ref{xxexa2.5}, $L_{\Lambda}(\Bbbk)$ does satisfy both the Dixmier and $R$-Dixmier property for any connected graded Poisson domain $R$. On the other hand, $L_{\Lambda}(\Bbbk)$ is not $\Bbbk[t^{\pm 1}]$-Dixmier since there is a Poisson algebra morphism $$\phi \colon L_{\Lambda}(\Bbbk)\longrightarrow L_{\Lambda}(\Bbbk) \otimes \Bbbk[t^{\pm 1}]$$ which sends $x_1^{\pm 1}$ to $x_1^{\pm 1}\otimes t^{\pm 1}$ and $x_2^{\pm 1}$ to $x_2^{\pm 1}\otimes t^{\pm 1}$ such that ${\rm{im}}(\phi)\not\subseteq L_{\Lambda}(\Bbbk) \otimes \Bbbk$.
\end{example}

\begin{example}
\label{xxexa1.7}
Let $A$ be a Poisson polynomial algebra 
$\Bbbk[x_1,\cdots,x_n]$ with a Poisson bracket
$\{-,-\}$. We can consider $A$ as a connected graded
algebra by setting $\deg x_i=1$ for all $i$. 
For any $d\geq 2$, let $A_{\geq d}$
be the space $\oplus_{i\geq d} A_i$. One 
can  check that $\{A_{\geq d}, A_{\geq d}\}
\subseteq A_{\geq 2d-2}\subseteq A_{\geq d}$, see for example \cite [Proposition 1.6.]{Possionstructers}.
So $A(d)\colon =\Bbbk\oplus A_{\geq d}$ is a finite codimensional Poisson
subalgebra of $A$.
\end{example}

One would wonder if $L_{\Lambda}(\Bbbk)$ has proper finite codimensional Poisson subalgebras. We have the following result. 

\begin{proposition}
 $L_{\Lambda}(\Bbbk)$ has no proper Poisson subalgebra of finite codimension over $\kk$.
\end{proposition}

\begin{proof}
We write a proof for the case $n = 2$. The proof for $n > 2$ is similar. 
Write monomials $x^m:=x_1^{m_1}x_2^{m_2}$ for $m=(m_1,m_2)\in\mathbb Z^2$. The bracket on monomials is
\[
\{x^u,x^v\}= \lambda_{12}\,(u_1v_2-u_2v_1)\,x^{u+v}\qquad(u,v\in\mathbb Z^2).
\]

Suppose $S\subseteq A$ is a Poisson subalgebra of finite codimension and $S\neq A$. Let
\[
M:=\{m\in\mathbb Z^2:\;x^m\notin S\}.
\]
Finite codimension implies $M$ is finite. Let
\[
T:=\mathbb Z^2\setminus M=\{m\in\mathbb Z^2:\;x^m\in S\},
\]
so $T$ is cofinite and closed under addition (because $S$ is a subalgebra).

Choose an integer $N>0$ such that $(N,0)\in T$ and $(0,N)\in T$ (this is possible since $M$ is finite). Put
\[
u:=(N,0),\qquad v:=(0,N).
\]
Fix any $m\in\mathbb Z^2$. Since $M$ is finite, there are only finitely many integers $n$ for which either $nu\in M$ or $m-nu\in M$. Hence, we can choose a non-zero integer $n$ such that both $nu\in T$ and $m-nu\in T$. Because $S$ is Poisson-closed, $x^{nu},x^{m-nu}\in S$ implies
\[
\{x^{nu},x^{m-nu}\}=\lambda_{12}\big((nN)(m_2)-0\cdot(m_1-nN)\big)\,x^{m}=\lambda_{12} nN m_2\,x^m\in S.
\]
As $\operatorname{char}\kk=0$ and $\lambda_{12}\neq0$, the scalar $nN m_2$ is zero only if $m_2=0$. Thus for every $m$ with $m_2\neq0$ we get $x^m\in S$.

If $m_2=0$, repeat the same argument exchanging the roles of $u$ and $v$: choose $n$ with $nv\in T$ and $m-nv\in T$, then
\[
\{x^{nv},x^{m-nv}\}= \lambda_{12}\,nN m_1\,x^m\in S,
\]
and again, the coefficient is nonzero unless $m_1=0$.

Finally, the remaining case $m=(0,0)$ corresponds to the unit $x^0=1\in S$ since $S$ is a subalgebra. Therefore every monomial $x^m$ lies in $S$, so $M=\varnothing$ and $S=A$, contradicting $S\neq A$.
\end{proof}

\section{Preliminaries}
\label{xxsec2}
In this section, we will review some invariants that will be used later on in this paper and establish some preparatory results.

\subsection{$u$-invariants}
\label{xxsec2.1}
In this subsection, we will recall the definition of $u$-invariant that was introduced in \cite{HNWZ1}. Note that in the reference \cite{HNWZ1}, the $u$-invariant was called the {\it set of universally degree 0 elements}.

Let $A$ be an algebra (not necessarily commutative). A 
{\it negatively-${\mathbb N}$-filtration} of $A$ is a descending sequence of subspaces ${\mathbf F}\colon =\{F_{-i}(A)\subseteq A\}_{i\geq 0}$ satisfying 
\begin{enumerate}
\item[(F1)]
$F_{-i-1}(A)\supseteq F_{-i}(A)\supseteq \cdots \supseteq 
F_{0}(A) \supseteq \Bbbk$ for all $i\geq 0$,
\item[(F2)]
$A=\bigcup_{i} F_{-i}(A)$,
\item[(F3)]
$F_{-i}(A) F_{-j}(A)\subseteq F_{-(i+j)} (A)$ 
for all $i,j\geq 0$.
\end{enumerate}
Note that, in \cite{HNWZ1}, the authors used ascending filtrations instead of descending ones. But 
these two concepts are equivalent after the index $-i$ is switched to $i$. The reason we use descending filtrations 
here is to match up with the definition of Poisson valuations as introduced and studied in \cite{HTWZ1, HTWZ2} that 
we shall recall soon.

Given a negatively-${\mathbb N}$-filtration ${\mathbf F}$, the corresponding associated graded ring is defined to be
$$\gr_{\mathbf F} A\colon =
\bigoplus_{i=0}^{-\infty} 
F_{i}(A)/F_{i+1}(A),$$
where $F_{i}(A)=0$ for $i\geq 1$ by convention. So $\gr_{\mathbf F}A$ is 
negatively-${\mathbb N}$-graded. We say 
${\mathbf F}$ is a {\sf good filtration} 
if $\gr_{\mathbf F} A$ is a domain (it is necessary 
that $A$ itself is a domain). For an 
algebra $A$, we set

\begin{equation}
\label{F2.1.1}\tag{F2.1.1}
\Phi_{good}(A)
={\text{the set of all good 
negatively-${\mathbb N}$-filtrations of $A$}}.
\end{equation}

According to \cite[Definition 0.2]{HNWZ1}, an algebra $A$ is called a {\it firm} domain if for every domain 
$B$, $A\otimes B$ is a domain. We say ${\mathbf F}$ is a {\it firm filtration} if $\gr_{\mathbf F} A$ 
is a firm domain (it is necessary that $A$ itself is a firm domain). For any firm domain $A$, let us set

\begin{equation}
\label{F2.1.2}\tag{F2.1.2}
\Phi_{firm}(A)
={\text{the set of all firm 
negatively-${\mathbb N}$-filtartions of $A$}}.
\end{equation}

Next, we recall the definitions of some invariants, defined in terms of the set of good or firm filtrations on $A$.

\begin{definition}
\cite[Definition 0.3]{HNWZ1}
\label{xxdef2.1}
Let $A$ be a domain. The {\it 
$u_{good}$-invariant} of $A$ is defined 
to be
$$u_{good}(A): = \bigcap_{{\mathbf F}\in 
\Phi_{good}(A)} F_0(A)$$
where the intersection $\bigcap$ runs over 
all filtrations in  $\Phi_{good}(A)$. In 
\cite{HNWZ1}, $u_{good}(A)$ is called 
{\it the set of universally degree $0$ 
elements} of $A$. Similarly, when $A$ is a 
firm domain, the {\it 
$u_{firm}$-invariant} of $A$ is defined 
to be
$$u_{firm}(A): = \bigcap_{{\mathbf F}\in 
\Phi_{firm}(A)} F_0(A)$$
where the intersection $\bigcap$ runs over 
all filtrations in  $\Phi_{firm}(A)$. 
 
\end{definition}

It is clear that $u_{good}(A)$ (resp. $u_{firm}(A)$) is a subalgebra of $A$. We 
say a filtration ${\mathbf F}$ is {\sf trivial} if $F_0(A)=A$.

\begin{definition}
\label{xxdef2.2}
Let $A$ be a domain.
\begin{enumerate}
\item[(1)]
We say $A$ is {\it $u_{good}$-maximal} if 
$u_{good}(A)=A$, or equivalently, $\Phi_{good}(A)$ 
consists of only the trivial filtration. 
\item[(2)]
We say $A$ is {\it $u_{good}$-minimal} if $u_{good}(A)=\Bbbk$.
\end{enumerate}
If $A$ is a firm domain, then the notions of 
{\it $u_{firm}$-maximal} and {\it $u_{firm}$-minimal}
are defined similarly.
\end{definition}

Below is a collection of basic properties on $u$-invariants which were proved earlier in \cite{HNWZ1}.

\begin{lemma}
\label{xxlem2.3} Let $A$ and $B$ be two domains.
\begin{enumerate}
\item[(1)]
If $A$ is a subalgebra of $B$, then $u_{good}(A)\subseteq u_{good}(B)$.
\item[(2)]
Suppose that $A$ and $B$ are commutative domains and that $B$ is a firm domain. Then 
$u_{good}(A) \otimes u_{good}(B)\
\subseteq u_{good}(A\otimes B)
\subseteq u_{good}(A) \otimes 
u_{firm}(B)$.
\item[(3)]
Suppose that $A$ and $B$ are commutative domains and the base field $\Bbbk$ is algebraically closed. Then 
$u_{good}(A\otimes B)=u_{good}(A)
\otimes u_{good}(B)$.
\item[(4)] If $B$ is a firm domain, then 
 $u_{good}(B)
\subseteq u_{firm}(B)$. 
\end{enumerate}
\end{lemma}

\begin{proof}
(1) This follows from the fact that every good 
negatively-${\mathbb N}$-filtration of $B$ 
induces a good negatively-${\mathbb N}$-filtration 
on $A$ by restriction, see \cite[Lemma 1.1(6)]{HNWZ1}.

(2) This is \cite[Proposition 2.3(1)]{HNWZ1}.

(3) When $\Bbbk$ is algebraically closed, 
$A\otimes B$ is a domain if and only if both $A$ and $B$ are domains \cite[Lemma 1.1]{RS}. Then $u_{firm}$ agrees with $u_{good}$. So the assertion follows from part (2) or \cite[Proposition 2.3(2)]{HNWZ1}. 

(4) Clear.
\end{proof}

Now we consider Poisson algebras (which are commutative as algebras). Let $A$ be a Poisson subalgebra of a 
Poisson algebra $B$. We define 
$$\Aut_{Poi.alg}(B\mid A)=\{\sigma \in \Aut_{Poi.alg}(B) \mid \sigma (a)=a,\; \forall \; a\in A\}.$$

\begin{theorem}
\label{xxthm2.4}
Let $A$ be a Poisson domain and $R$ be a Poisson firm 
domain. Suppose that $A$ is $u_{good}$-maximal and 
that $R$ is $u_{firm}$-minimal. 
\begin{enumerate}
\item[(1)]
Then, for 
every injective Poisson algebra morphism
$f: A\to A\otimes R$, 
${\rm{im}}(f)\subseteq A\otimes \Bbbk\subseteq A\otimes R$.
\item[(2)]
Let $C$ be a Poisson subalgebra of $A$.
Then there is a short exact sequence of groups
$$1\to \Aut_{Poi.alg}(A\otimes R\mid A)\to 
\Aut_{Poi.alg}(A\otimes R\mid C)\to 
\Aut_{Poi.alg}(A\mid C)\to 1.$$
\item[(3)]
If further $A$ is Dixmier, then $A$ is $R$-Dixmier.
\end{enumerate}
\end{theorem}

\begin{proof}
(1) Let $f\colon A\longrightarrow A\otimes R$ be an injective Poisson 
algebra morphism and let $\overline{A}$ be the image
of $f$. Then $\overline{A}$ is a Poisson subalgebra 
of $A\otimes R$. Then we have
$$\begin{aligned}
\overline{A}& =u_{good}(\overline{A}) \qquad\qquad\qquad \qquad 
{\text{as $A\cong \overline{A}$ is 
$u_{good}$-maximal}}\\
&\subseteq u_{good}(A\otimes R) \qquad\qquad \qquad
{\text{by Lemma \ref{xxlem2.3}(1)}}\\
&\subseteq u_{good}(A)\otimes u_{firm}(R) \;\qquad\;\;
{\text{by Lemma \ref{xxlem2.3}(2)}}\\
&= u_{good}(A)\otimes \Bbbk \qquad \; \qquad \qquad
{\text{as $R$ is $u_{firm}$-minimal}}\\
&=A\otimes \Bbbk \qquad \qquad \; \qquad\qquad\quad
{\text{as $A$ is 
$u_{good}$-maximal.}}
\end{aligned}$$

(2) Let $\phi$ be in $\Aut_{Poi.alg}(A\otimes R\mid C)$ and let $f=\phi\mid_{A\otimes \Bbbk}\colon A\otimes \Bbbk \longrightarrow A\otimes R$ be the restriction of $\phi$ to $A\otimes \Bbbk$ that is identified with $A$. By 
part (1), ${\rm{im}}(f)\subseteq A\otimes \Bbbk$. So we can consider $f$ as a Poisson algebra endomorphism of $A$. Since $\phi$ is invertible, we can apply part (1) to the inverse of $\phi$. Hence $f$ is bijective. This implies that there is a map from $\Aut_{Poi.alg}(A\otimes R\mid C)\to \Aut_{Poi.alg}(A\mid C)$ when $C$ is identified with $C\otimes \Bbbk$.
Now it is routine to check that the short sequence is exact.

(3) By the proof of part (1), every injective Poisson algebra morphism
$f: A\to A\otimes R$,  induces
an injective map $\pi\colon A\stackrel{f}{\cong} \overline{A}\hookrightarrow 
A\otimes \Bbbk\cong A$. Since $A$ is Dixmier, 
$\pi$ is a bijection, which implies that 
${\rm{im}}(f)=\overline{A}=A\otimes \Bbbk$ as required. Therefore, $A$ is $R$-Dixmier.
\end{proof}

\subsection{Relative Dixmier property for uniparameter simple Poisson torus} In this subsection, we establish some results 
concerning the relative Dixmier property for some uniparameter Poisson torus $L_{\Lambda}(\Bbbk)$. First, we start with an example. 
\begin{example}
\label{xxexa2.5}
Let $A$ be the Poisson torus $\Bbbk[x^{\pm 1}, y^{\pm 1}]$ with Poisson bracket determined by 
$$\{x,y\}=q xy$$ for some $0\neq q\in \Bbbk$. It is easily checked that $A$ is a simple Poisson domain (for example, one can check condition~(3) of Lemma~\ref{xxlem2.7}). 
We claim that $A$ is Dixmier. To see this, we let $\phi\colon A\longrightarrow A$ be a Poisson algebra 
morphism. Note that $\phi$ preserves invertible elements and all invertible elements
in $A$ are of the form $cx^iy^j$ for some $c\in \Bbbk^{\times}$ and $i,j\in {\mathbb Z}$. 
Let $\phi(x)=a x^{i}y^j$ and $\phi(y)=b x^{f} y^g$
where $a,b\in \Bbbk^{\times}$ are nonzero scalars and $i, j , f, g\in \mathbb Z$. Then applying $\phi$ to $\{x,y\}=q xy$ leads to 
$\{x^i y^j, x^f y^g\}=q x^{i+f}y^{j+g}$. 
By an easy computation, we shall have 
$$ \{x^i y^j, x^f y^g\}= (ig-jf) q x^{i+f}y^{j+g}$$
which implies that $ig-jf=1$. So $\phi$ is invertible, indeed we have that $ \phi (c_1 x^{g} y^{-j}) = x$ and 
$ \phi (c_2 y^{i} x^{-f}) = y$
for some scalars $c_1, c_2\in \Bbbk^{\times}$. Therefore, the Poisson algebra $A$ is Dixmier. By \cite[Corollary 1.4]{HNWZ1}, 
$A$ is also $u_{good}$-maximal. If $R$ is a connected graded firm Poisson domain, then by Theorem \ref{xxthm2.4}, $A$ is also $R$-Dixmier. Note that the Dixmier property of $A$ also follows from \cite[Thorem 3.7]{CO}
\end{example}

In general, the above direct verification won't work for the Poisson torus $L_{\Lambda}(\Bbbk)$ when $n\geq 3$. Fortunately, we are able to adapt the proof of Richard \cite{Ric} concerning the Dixmier property for the simple uniparameter quantum torus to establish the following Poisson analogue.

\begin{theorem}
\label{xxthm2.6}
Let $\Lambda =\lambda M$ for some $\lambda \in \Bbbk^{\times}$ and a skew-symmetric matrix $M=(m_{ij})\in M_{n}(\mathbb{Z})$ with integer entries such that $L_{\Lambda}(\Bbbk)$ is a simple Poisson torus. Then each Poisson endomorphism of $L_{\Lambda}(\Bbbk)$ is an automorphism. That is, such a simple Poisson torus $L_{\Lambda}(\Bbbk)$ is Dixmier and thus $R$-Dixmier for any connected graded Poisson firm domain $R$. 
\end{theorem} 

To proceed, we first need to establish two auxiliary lemmas. The first one is a well-known observation, whose proof can be found in \cite{GL, Van}. We state it with a sketchy proof.  
\begin{lemma}
\label{xxlem2.7}
The following statements are equivalent for 
$L_{\Lambda}(\Bbbk)$
with respect to the Poisson bracket determined by~$\Lambda$.
\begin{enumerate}
\item [(1)] The Poisson torus $L_{\Lambda}(\Bbbk)$ is simple as a Poisson algebra;
\item [(2)] The Poisson center of $L_{\Lambda}(\Bbbk)$ is reduced to the base field $\Bbbk$;
\item [(3)] There does not exist a tuple $a=(a_{1}, \cdots, a_{n}) \in \Z^{n}\setminus \{0=(0, \cdots, 0)\}$ 
such that
\[
\sum_{i=1}^{n} a_{i} \lambda_{ij}=0, \quad \forall 1\leq j\leq n.
\]
\end{enumerate}
\end{lemma}
\begin{proof}
The equivalence of (1) and (2) is due to Vancliff \cite[Lemma 1.2]{Van}, see also \cite[Corollary 3.2]{GL}. Since $\Bbbk[x_{1}^{\pm 1}, \cdots, x_{n}^{\pm 1}]$ is $\mathbb{Z}^{n}$-graded Poisson algebra with ${\rm deg}\, (x_{i})=e_{i}=(0, \cdots, 1, \cdots, 0)$ for $i=1, \cdots, n$, each Poisson-central element in $L_{\Lambda}(\Bbbk)$ is a sum of scalar multiples of monomials in $x_{1}^{\pm 1}, \cdots, x_{n}^{\pm 1}$, each of which is Poisson-central. Without loss of generality, we may assume that each Poisson-central element of $L_{\Lambda}(\Bbbk)$ is of the form: $x=\prod_{i=1}^{n} x_{i}^{a_{i}}$ with $a_{i}\in \mathbb{Z}$. Then we shall have $\{x, x_{j}\}=0$ for $j=1, \cdots, n$. As a result, we have
\[
\sum_{i=1}^{n} a_{i}\lambda_{ij} =0
\]
for $j=1, \cdots, n$. Now the equivalence of (2) and (3) follows. 
\end{proof}

\begin{lemma}
\label{xxlem2.8}
Let $A=L_{\Lambda}(\Bbbk)=\Bbbk_{\Lambda}[x_{1}^{\pm 1}, \cdots, x_{n}^{\pm 1}]$ be a simple Poisson torus of rank $n$ under 
the Poisson bracket determined by $\Lambda$. The following statements are equivalent.
\begin{enumerate}
\item [(1)] Every Poisson endomorphism of $A$ is an automorphism;
\item [(2)] Let $ \mathcal{Q} \subseteq A$ be a Poisson subalgebra of $A$ such that $\mathcal{Q}$ is isomorphic to $A$ as a Poisson algebra. Then $ \mathcal{Q} =A$;
\item [(3)] For any skew-symmetric matrix $\Lambda^{\prime}=(\lambda_{ij}^{\prime})\in M_{n}(\Bbbk)$ such that $A=L_{\Lambda^{\prime}}(\Bbbk)=\Bbbk_{\Lambda^{\prime}}[z_{1}^{\pm 1}, \cdots, z_{n}^{\pm 1}]$ 
as a Poisson torus, if there exist an invertible integer matrix $G=(g_{ij}) \in {\rm GL}_{n}(\Z)$ and positive 
integers $p_{1}, \cdots, p_{n}$ such that
\[
\lambda_{ij}^{\prime}=\sum_{1\leq k, t\leq n} p_{k}p_{t} g_{ki} g_{tj}\lambda_{kt}^{\prime}, \quad \forall 1\leq i, j\leq n,
\]
then $p_{1}= \cdots= p_{n}=1$. 
\end{enumerate}
\end{lemma}
\begin{proof} 
$(1) \Rightarrow (2)$: Suppose that $\mathcal{Q} \subseteq A$ is a Poisson subalgebra of the simple Poisson torus $A=\Bbbk_{\Lambda}[x_{1}^{\pm 1}, \cdots, x_{n}^{\pm 1}]$ 
such that $A\stackrel{\phi}{\cong} \mathcal{Q}$ as Poisson algebras. Denote the composition of the embedding of $\mathcal{Q}$ into $A$ with the isomorphism 
$\phi$ by 
\[
\psi: A\stackrel{\phi}{\longrightarrow} \mathcal{Q} \hookrightarrow A.
\]
Obviously, $\psi$ is an (injective) Poisson endomorphism of $A$. By (1), we know that $\psi$ is an automorphism, which implies that $\mathcal{Q}=A$. 

$(2) \Rightarrow (1)$: Let $\phi \colon A \longrightarrow A$ be a Poisson endomorphism of $A$. Since $A$ is a simple Poisson torus, we know that $\phi$ is injective. 
As a result, $A$ is isomorphic to $\mathcal{Q}\colon =\phi(A)$ as a Poisson algebra, and $\mathcal{Q}\subseteq A$ is a Poisson subalgebra of $A$. By (2), we know that 
$\mathcal{Q}=A$, which implies that $\phi$ is an automorphism. 

$(2) \Rightarrow (3)$: Suppose $\Lambda^{\prime}=(\lambda_{ij}^{\prime})\in M_{n}(\Bbbk)$ is a skew-symmetric matrix such that 
$A=\Bbbk_{\Lambda^{\prime}}[z_{1}^{\pm 1}, \cdots, z_{n}^{\pm 1}]$ and there exist an invertible integer matrix 
$G=(g_{ij}) \in {\rm GL}_{n}(\Z)$ together with positive integers $p_{1}, \cdots, p_{n}$ such that
\[
\lambda_{ij}^{\prime}=\sum_{1\leq k,t\leq n} p_{k}p_{t} g_{ki}g_{tj} \lambda_{kt}^{\prime}, \quad \forall 1\leq i, j\leq n.
\]
Note that we have $\{z_{i}, z_{j}\}=\lambda_{ij}^{\prime} z_{i}z_{j}$ for $1\leq i, j\leq n$. Let $\mathcal{Q}$ be the Poisson subalgebra of $A$ generated by 
$y_{1}^{\pm 1}\colon =z_{1}^{\pm p_{1}}, \cdots, y_{n}^{\pm 1}\colon =z_{n}^{\pm p_{n}}$. Since $z_{1}, \cdots, z_{n}$ 
are algebraically independent and $p_{i}\geq 1$, we know that $\mathcal{Q}$ is indeed a Poisson torus of rank 
$n$ generated by $y_{1}^{\pm 1}, \cdots, y_{n}^{\pm 1}$ subject to the following Poisson bracket: 
\[
\{y_{i}, y_{j}\}=\{z_{i}^{p_{i}}, z_{j}^{p_{j}}\}=p_{i}p_{j} \lambda_{ij}^{\prime} y_{i}y_{j}
\]
for $1\leq i, j\leq n$. Let us define an algebra morphism
\[
\phi \colon A=\Bbbk[z_{1}^{\pm 1}, \cdots, z_{n}^{\pm 1}]\longrightarrow \mathcal{Q}=\Bbbk[y_{1}^{\pm 1}, \cdots, y_{n}^{\pm 1}]
\]
by setting
\[
\phi(z_{i})=\prod_{k=1}^{n} y_{k}^{g_{ki}}
\]
for $i=1, \cdots, n$. It is easy to check that 
\begin{eqnarray*}
\{\phi(z_{i}), \phi(z_{j})\} &=&\{\prod_{k=1}^{n} y_{k}^{g_{ki}}, \prod_{t=1}^{n} y_{t}^{g_{tj}}\}\\
&=& (\sum_{k, t=1}^{n} p_{k}p_{t} g_{ki} g_{tj} \lambda_{kt}^{\prime})(\prod_{k=1}^{n} y_{k}^{g_{ki}}) (\prod_{t=1}^{n} y_{t}^{g_{tj}})\\
&=& \lambda_{ij}^{\prime} (\prod_{k=1}^{n} y_{k}^{g_{ki}}) (\prod_{t=1}^{n} y_{t}^{g_{tj}})\\
&=& \lambda_{ij}^{\prime} \phi(z_{i}) \phi(z_{j})\\
&=& \phi(\lambda_{ij}^{\prime} z_{i} z_{j})\\
&=& \phi(\{z_{i}, z_{j}\})
\end{eqnarray*}
for any $1\leq i, j\leq n$. Therefore, $\phi$ is a Poisson algebra morphism. Since $G=(g_{ij})\in {\rm GL}_{n}(\Z)$ is an invertible 
integer matrix, $\phi$ is indeed a Poisson isomorphism. By (2), we know that 
\[
\Bbbk[z_{1}^{\pm p_{1}}, \cdots, z_{n}^{\pm p_{n}}]=\mathcal{Q}=A=\Bbbk[z_{1}^{\pm 1}, \cdots, z_{n}^{\pm 1}].
\]
As a result, we must have that $p_{1}=\cdots =p_{n}=1$.  

$(3) \Rightarrow (2)$: Suppose that $\mathcal{Q}\subseteq A=L_{\Lambda}(\Bbbk)=\Bbbk_{\Lambda}[x_{1}^{\pm 1}, \cdots, x_{n}^{\pm 1}]$ 
is a Poisson subalgebra of $A$ and $A$ is isomorphic to $\mathcal{Q}$ as a Poisson algebra. Therefore, there is an isomorphism of Poisson algebras:
\[
\phi \colon A=\Bbbk[x_{1}^{\pm 1}, \cdots, x_{n}^{\pm 1}] \longrightarrow \mathcal{Q}\subseteq A.
\]
Since $A=\Bbbk[x_{1}^{\pm 1}, \cdots, x_{n}^{\pm 1}]$ is generated by $x_{1}^{\pm 1}, \cdots, x_{n}^{\pm 1}$ and $\phi(x_{i})$ is a monomial 
in $\Bbbk[x_{1}^{\pm 1}, \cdots, x_{n}^{\pm 1}]$ for each $i=1, \cdots, n$, we know that  $\mathcal{Q}$ is generated by the monomials 
$y_{i}^{\pm 1}\colon =[\phi(x_{i})]^{\pm 1}=x^{\pm b_{i}}=(\prod_{k=1}^{n} x_{k}^{b_{ki}})^{\pm 1}$ 
for some $b_{i}=(b_{1i}, \cdots, b_{ni}) \in \mathbb{Z}^{n}$ with $i=1, \cdots, n$. As a result, the subset $\mathbb{S}=\{b\in \Z^{n}\mid x^{b}\in \mathcal{Q}\}$ 
is actually a subgroup of $\Z^{n}$. Since $\phi$ is an isomorphism of algebras, we must have that $\mathbb{S}$ is a free abelian group of the same rank as $\Z^{n}$. 
Now we can choose a basis $\{c_{1}, \cdots, c_{n}\}$ of $\Z^{n}$ such that $\{p_{1} c_{1}, \cdots,  p_{n} c_{n}\}$ is a basis of $\mathbb{S}$ (whose rank is equal to $n$) for 
some positive integers $p_{1}\mid p_{2}\mid \cdots \mid p_{n}$. Setting $z_{i}^{\pm 1}\colon =x^{\pm c_{i}}=(\prod_{k=1}^{n} x_{k}^{c_{ki}})^{\pm 1}$ for $i=1, \cdots, n$,  
we know that $A$ is the same as the Poisson torus $L_{\Lambda^{\prime}}(\Bbbk)=\Bbbk_{\Lambda^{\prime}}[z_{1}^{\pm 1}, \cdots, z_{n}^{\pm 1}]$, which is generated by 
$z_{1}^{\pm 1}, \cdots, z_{n}^{\pm 1}$ subject to the following Poisson bracket
\[
\{z_{i}, z_{j}\}=\lambda^{\prime}_{ij}z_{i}z_{j}
\]
where $\Lambda^{\prime}=(\lambda^{\prime}_{ij})$ with $\lambda_{ij}^{\prime} :=\sum_{1\leq k,t\leq n} c_{ki}c_{tj} \lambda_{kt}$ for $1\leq i, j\leq n$.  
As a result, $\mathcal{Q}$ is indeed the Poisson torus $\Bbbk[(z_{1}^{p_{1}})^{\pm 1}, \cdots, (z_{n}^{p_{n}})^{\pm 1}]$ generated by 
$z_{1}^{\pm p_{1}}, \cdots, z_{n}^{\pm p_{n}}$ subject to the following Poisson bracket: 
\[
\{z_{i}^{p_{i}}, z_{j}^{p_{j}}\}=p_{i}p_{j}\lambda_{ij}^{\prime}z_{i}^{p_{i}}z_{j}^{p_{j}},\quad \forall 1\leq i, j\leq n.
\]
Since the Poisson torus $L_{\Lambda^{\prime}}(\Bbbk)=\Bbbk_{\Lambda^{\prime}}[z_{1}^{\pm 1}, \cdots, z_{n}^{\pm 1}]$ (the same as $A$) is also isomorphic to the Poisson torus 
$\mathcal{Q}=\Bbbk[z_{1}^{\pm p_{1}}, \cdots, z_{n}^{\pm p_{n}}]$ as a Poisson algebra via $\phi\colon A \longrightarrow \mathcal{Q}$, there 
exists an invertible integer matrix $G=(g_{ij})\in {\rm GL}_{n}(\Z)$ such that $\phi(z_{i})=\alpha_{i} \prod_{k=1}^{n}  z_{k}^{p_{k} g_{ki}}$ where 
$\alpha_{i} \in \Bbbk^{\times}$ for $i=1, \cdots, n$. Therefore, we have that 
\[
\lambda_{ij}^{\prime}=\sum_{k, t=1}^{n} p_{k}p_{t} g_{ki} g_{tj} \lambda_{kt}^{\prime}
\]
for any $1\leq i, j\leq n$. Since $A=L_{\Lambda^{\prime}}(\Bbbk)=\Bbbk_{\Lambda^{\prime}}[z_{1}^{\pm 1}, \cdots, z_{n}^{\pm 1}]$ and $G$ is in ${\rm GL}_{n}(\mathbb{Z})$, we must 
have that $p_{1}=\cdots=p_{n}=1$ using (3). Therefore, we have that $\mathbb{S}=\Z^{n}$, which implies that $\mathcal{Q}=\Bbbk[x^{\pm c_{1}}, \cdots, x^{\pm c_{n}}]=\Bbbk[x_{1}^{\pm 1}, \cdots, x_{n}^{\pm 1}]=A$. 
\end{proof}

\begin{proof}[Proof of Theorem \ref{xxthm2.6}.] Note that $\Lambda =\lambda M$ for some $\lambda \in \Bbbk^{\times}$ and a skew-symmetric matrix $M=(m_{ij})\in M_{n}(\mathbb{Z})$ 
with integer entries.  Recall that $L_{\Lambda}(\Bbbk)$ is the Poisson torus generated 
by $x_{1}^{\pm 1}, \cdots, x_{n}^{\pm 1}$ subject to the following Poisson bracket: 
\[
\{x_{i}, x_{j}\}=\lambda_{ij} x_{i}x_{j}=\lambda m_{ij}x_{i}x_{j}
\] 
for $1\leq i\leq j\leq n$. Since the Poisson torus $L_{\Lambda}(\Bbbk)$ is Poisson simple, by Lemma \ref{xxlem2.7}, we 
know that the matrix $\Lambda$ is non-singular, which is equivalent to the fact that the matrix $M=(m_{ij})$ is non-singular. That is, ${\rm det}(M)\neq 0$. 

To show that $L_{\Lambda}(\Bbbk)$ is Dixmier, we prove that Condition~(3) of Lemma~\ref{xxlem2.8} holds.
Suppose there exist a skew-symmetric matrix $\Lambda^{\prime}=(\lambda_{ij}^{\prime})\in M_{n}(\Bbbk)$ such that $L_{\Lambda}(\Bbbk)=\Bbbk_{\Lambda^{\prime}}[z_{1}^{\pm 1}, \cdots, z_{n}^{\pm 1}]$ 
and an invertible integer matrix $G=(g_{ij})\in {\rm GL}_{n}(\mathbb{Z})$ together with positive integers $p_{1}, \cdots, p_{n}$ such that 

\begin{equation}\label{F2.2.1}
\tag{F2.2.1}
\lambda_{ij}^{\prime} =\sum_{k, t=1}^{n} p_{k} p_{t} g_{ki} g_{tj} \lambda_{kt}^{\prime}
\end{equation}

for $1\leq i, j\leq n$. We shall prove that $p_{1}=\cdots =p_{n}=1$. 

Since $\Bbbk_{\Lambda}[x_{1}^{\pm 1}, \cdots, x_{n}^{\pm 1}]=\Bbbk_{\Lambda^{\prime}}[z_{1}^{\pm 1}, \cdots, z_{n}^{\pm 1}]$ as Poisons algebras, 
there exists an invertible integer matrix $T=(t_{ij})\in {\rm GL}_{n}(\mathbb{Z})$ such that
\[
z_{i}=\prod_{k=1}^{n} x_{k}^{t_{ki}}, 
\] 
for $i=1, \cdots, n$ and so from 
$ \{z_i , z_j \} = \{ \prod_{k=1}^{n} x_{k}^{t_{ki}}, 
\prod_{l=1}^{n} x_{l}^{t_{lj}}
\} $, we have that 
$\Lambda^{\prime} =T^{t} \Lambda T$, where 
$T^t$
denotes the matrix transpose of $T$.
As a result, we have $\Lambda^{\prime}=T^{t} (\lambda M) T=\lambda (T^{t} M T)$. 
Setting $M^{\prime}=(m_{ij}^{\prime})=T^{t} M T$, then we have $\Lambda^{\prime}=\lambda M^{\prime}$ or equivalently 
$\lambda_{ij}^{\prime}=\lambda m_{ij}^{\prime}$ for $1\leq i, j\leq n$. Since ${\rm det}(M)\neq 0$, we have that ${\rm det}(M^{\prime})={\rm det}(T^{t} M T)=({\rm det}(T))^{2} {\rm det} (M)\neq 0$.  
This also implies that $\det(\Lambda') \neq 0.$
Now   
formula~(\ref{F2.2.1}) gives us 
\[
\Lambda' = G^t D \Lambda' D G, \quad \text{where } D = \mathrm{Diag}[p_1, \dots, p_n].
\]  
and consequently $
1 = {\det(G)}^2 (\prod_{i=1}^{n} p_{i})^{2}$. 
Because $G$ has entries in $\mathbb{Z}$, we conclude that $
\prod_{i=1}^{n} p_i = 1,$
and therefore  
\[
p_1 = \cdots = p_n = 1.
\]
By Lemma \ref{xxlem2.8}, we know that every Poisson endomorphism of $L_{\Lambda}(\Bbbk)$ is an automorphism. Thus, we have proved $L_{\Lambda}(\Bbbk)$ is Dixmier. 
Since $L_{\Lambda}(\Bbbk)$ is $u_{\mathrm{good}}$-maximal (see \cite[Corollary~1.4]{HNWZ1}) 
and $R$ is $u_{\mathrm{firm}}$-minimal, it follows from Theorem~\ref{xxthm2.4}(3) that 
$L_{\Lambda}(\Bbbk)$ is $R$-Dixmier.
\end{proof}

\subsection{Poisson $u$-invariants}
\label{xxsec2.3}
In this subsection, we define some secondary invariants associated with the $u$-invariants.

Let $A$ be a Poisson algebra with Poisson 
bracket $\{-,-\}$ (as an associative algebra, it is commutative). Let $w$ be an integer.
Recall from \cite[Definition 2.2]{HTWZ1} that
a {\it $w$-filtration} of $A$ is a descending sequence 
of subspaces ${\mathbf F}\colon =\{F_i(A)\subseteq A\}_{i\in {\mathbb Z}}$ satisfying 
\begin{enumerate}
\item[(F1)]
$F_{i}(A)\supseteq F_{i+1}(A)$ for all $i\in {\mathbb Z}$ and $1\in F_0(A)\setminus F_1(A)$,
\item[(F2)]
$A=\bigcup_{i\in {\mathbb Z}} F_{i}(A)$ and 
$0=\bigcap_{i\in {\mathbb Z}} F_{i}(A)$,
\item[(F3)]
$F_{i}(A) F_{j}(A)\subseteq F_{i+j} (A)$ 
for all $i,j\in {\mathbb Z}$,
\item[(F4)]
$\{F_i(A), F_j(A)\}\subseteq F_{i+j-w}(A)$ for all
$i,j\in {\mathbb Z}$.
\end{enumerate}
A $w$-filtration ${\mathbf F}$ is called 
{\it negatively-${\mathbb N}$-$w$-filtration} 
if $F_{i}=0$ for all $i\geq 1$. It is clear that
a negatively-${\mathbb N}$-$w$-filtration is
a special kind of a 
negatively-${\mathbb N}$-filtration which 
is defined in the previous subsection.

Given a $w$-filtration ${\mathbf F}$, the associated graded ring is defined to be
$$\gr_{\mathbf F} A\colon =\bigoplus_{i\in {\mathbb Z}} 
F_{i}(A)/F_{i+1}(A).$$
Following \cite[Definition 2.1]{HTWZ1}, that a Poisson algebra $P$ is called {\it Poisson $w$-graded} if
$P=\oplus_{i\in {\mathbb Z}} P_i$ is a ${\mathbb Z}$-graded 
commutative algebra, and  
$\{P_i, P_j\}\subseteq P_{i+j-w}$ for all $i,j\in {\mathbb Z}$.
 Clearly when 
 ${\mathbf F}$
 is a $w$-filtration, then 
 $\gr_{\mathbf F}A$ is a $w$-graded Poisson algebra. 
Recall that ${\mathbf F}$ is a {\sf good filtration} 
if $\gr_{\mathbf F} A$ is a domain. For a Poisson algebra $A$ and
an integer $w$, set

\begin{equation}
\label{F2.3.1}\tag{F2.3.1}
\Phi_{good,w}(A)
={\text{the set of all good negatively-${\mathbb N}$-$w$-filtartions of $A$}},
\end{equation}
and

\begin{equation}
\label{F2.3.2}\tag{F2.3.2}
\Phi_{firm,w}(A)
={\text{the set of all firm negatively-${\mathbb N}$-$w$-filtrations of $A$}}.
\end{equation}

Similar to Definition \ref{xxdef2.1}, we have the following definition.

\begin{definition}
\label{xxdef2.9}
Let $A$ be a Poisson domain. The {\it 
$u_{good,w}$-invariant} of $A$ is defined 
to be
$$u_{good,w}(A): = \bigcap_{{\mathbf F}\in 
\Phi_{good,w}(A)} F_0(A)$$
where the intersection $\bigcap$ runs over 
all filtrations in  $\Phi_{good,w}(A)$.
Similarly, when $A$ is a Poisson firm 
domain, the {\it $u_{firm,w}$-invariant} 
of $A$ is defined to be
$$u_{firm,w}(A): = \bigcap_{{\mathbf F}\in 
\Phi_{firm,w}(A)} F_0(A)$$
where the intersection $\bigcap$ runs over 
all filtrations in  $\Phi_{firm,w}(A)$.
\end{definition}

It is clear that $u_{good,w}(A)$ (resp. $u_{firm,w}(A)$) is a subalgebra of $A$. Since $\Phi_{good,w}(A)\subseteq \Phi_{good, w+1}(A)\subseteq \Phi_{good}(A)$, we have $u_{good,w}(A)\supseteq u_{good,w+1}(A)\supseteq u_{good}(A)$. It is unclear to us whether or not $u_{good}(A)=\bigcap_{w\in \Z} u_{good,w}(A)$.

\begin{definition}
\label{xxdef2.10}
Let $A$ be a Poisson domain and $w$ an integer.
\begin{enumerate}
\item[(1)]
We say $A$ is {\it $u_{good,w}$-maximal} if 
$u_{good,w}(A)=A$, or equivalently, 
$\Phi_{good,w}(A)$ consists of only the 
trivial filtration. 
\item[(2)]
We say $A$ is {\it $u_{good,w}$-minimal} if $u_{good,w}(A)=\Bbbk$.
\end{enumerate}
If $A$ is a firm Poisson domain, then 
the notions of {\it $u_{firm,w}$-maximal} 
and {\it $u_{firm,w}$-minimal}
are defined similarly.
\end{definition}

Let $B$ be the Poisson domain $P_{\Omega-\xi}$ in Example \ref{xxexa0.1} where $\Omega$ is 
an i.s. potential of degree $d\geq 5$, then, as we will see in Lemma \ref{xxlem2.14}, 
$B$ is $u_{good, (d-4)}$-maximal and $u_{good, (d-3)}$-minimal.

Similar to Lemma \ref{xxlem2.3} we have: 

\begin{lemma}
\label{xxlem2.11} Let $A$ and $B$ be two Poisson domains and $w$ be an integer.
\begin{enumerate}
\item[(1)]
If $A$ is a subalgebra of $B$, then $u_{good,w}(A)\subseteq u_{good,w}(B)$.
\item[(2)]
Suppose $B$ is a firm domain. Then 
$$u_{good,w}(A) \otimes u_{good,w}(B)\
\subseteq u_{good,w}(A\otimes B)
\subseteq u_{good,w}(A) \otimes 
u_{firm,w}(B).$$
\item[(3)]
Suppose $\Bbbk$ is algebraically closed. Then 
$u_{good,w}(A\otimes B)=u_{good,w}(A)
\otimes u_{good,w}(B)$.
\item[(4)]
If $B$ is a firm domain, then $u_{good,w}(B)
\subseteq u_{firm,w}(B)$.
\end{enumerate}
\end{lemma}

\begin{proof}
(1,3,4) See the proof of Lemma \ref{xxlem2.3}.

(2) We modify the proof of \cite[Proposition 2.3(1)]{HNWZ1} as below. Since $A\otimes \Bbbk$ is a Poisson subalgebra 
of $A\otimes B$, we have
$u_{good,w}(A)\otimes \Bbbk 
\subseteq u_{good,w}(A\otimes B)$ by part (1). 
Similarly, $\Bbbk\otimes u_{good,w}(B)\subseteq 
u_{good,w}(A\otimes B)$. Therefore 
$u_{good,w}(A)\otimes u_{good,w}(B)\subseteq 
u_{good,w}(A\otimes B)$.

For $u_{good,w}(A\otimes B)\subseteq u_{good,w}(A)
\otimes u_{firm,w}(B)$, let ${\mathbf F}$ be a 
good negatively-${\mathbb N}$-$w$-filtration of $A$ 
and ${\mathbf G}$ be a firm 
negatively-${\mathbb N}$-$w$-filtration of $B$. 
By \cite[Lemma 1.2(3)]{HNWZ1}, 
${\mathbf H}:={\mathbf F}\otimes {\mathbf G}$ is a 
good negatively-${\mathbb N}$-filtration of 
$A\otimes B$ with $H_0=F_0(A)\otimes G_0(B)$. 
By the definition of the product of Poisson algebras, 
{\small $$\{A_{i}\otimes B_j, A_{l}\otimes B_{k}\}
\subseteq \{A_{i}, A_{l}\}\otimes B_j B_{k}
+A_{i}A_{l}\otimes\{B_j,B_{k}\}
\subseteq A_{i+l-w}\otimes B_{j+k} + A_{i+l}\otimes B_{j+k-w}$$}
which implies that ${\mathbf H}$ is a
$w$-filtration. So 
$$\begin{aligned}
u_{good,w}(A\otimes B)&=\bigcap_{{\mathbf H}
\in \Phi_{good,w}(A\otimes B)} H_0\\
&\subseteq 
\bigcap_{{\mathbf H}={\mathbf F}\otimes {\mathbf G},
{\mathbf F}\in \Phi_{good,w}(A), {\mathbf G}\in \Phi_{firm,w}(B)} F_0(A)\otimes G_0(B)\\
&=[\bigcap_{{\mathbf F}\in \Phi_{good,w}(A)} F_0(A)
]\otimes 
[\bigcap_{{\mathbf G}\in \Phi_{firm,w}(B)}G_0(B)
]
\\
&=u_{good,w}(A)\otimes u_{firm,w}(B).
\end{aligned}
$$
\end{proof}

By Lemma \ref{xxlem2.11} and the proof of Theorem \ref{xxthm2.4} we obtain the following.

\begin{theorem}
\label{xxthm2.12}
Let $A$ be a Poisson domain and $R$ be a Poisson firm domain. Let $w$ be an integer. Suppose that $A$ is $u_{good,w}$-maximal and that $R$ is $u_{firm,w}$-minimal. 
\begin{enumerate}
\item[(1)]
Then, for 
every injective Poisson algebra morphism
$f\colon A\longrightarrow A\otimes R$, 
${\rm{im}}(f)\subseteq A\otimes \Bbbk\subseteq A\otimes R$.
\item[(2)]
Let $C$ be a Poisson subalgebra of $A$. Then there is a short exact sequence of groups
$$1\to \Aut_{Poi.alg}(A\otimes R\mid A)\to 
\Aut_{Poi.alg}(A\otimes R\mid C)\to 
\Aut_{Poi.alg}(A\mid C)\to 1.$$
\item[(3)]
If further $A$ is Dixmier, then $A$ is $R$-Dixmier.
\end{enumerate}
\end{theorem}

We will frequently use Theorem \ref{xxthm2.12} in the next subsection.

\subsection{Poisson valuations and relative Dixmier property}
\label{xxsec2.3}
In this subsection, we will apply Theorem \ref{xxthm2.12} to some Poisson algebras. First of all, we will recall a lemma as proved in \cite{HTWZ1} and the definition of Poisson $w$-valuations, which is indeed equivalent to the definition of $w$-filtrations given in the previous subsection. We start with a definition.

\begin{definition}
\label{xxdef2.13}
\cite[Definition 1.1]{HTWZ1} Let $A$ be a Poisson domain and $w$ be an integer. A map 
$$\nu\colon A\longrightarrow {\mathbb Z}\cup\{\infty\}$$
is called a {\it $w$-valuation} on $A$ if, for all 
$a,b\in A$,
\begin{enumerate}
\item[(1)]
$\nu(a)=\infty$ if and only if $a=0$,
\item[(2)]
$\nu(a)=0$ for all $a\in \Bbbk^{\times}:=\Bbbk\setminus \{0\}$,
\item[(3)]
$\nu(ab)=\nu(a)+\nu(b)$,
\item[(4)]
$\nu(a+b)\geq \min\{\nu(a),\nu(b)\}$, 
\item[(5)]
$\nu(\{a,b\})\geq \nu(a)+\nu(b)-w$.
\end{enumerate}
\end{definition}

By \cite[Lemma 2.6]{HTWZ1}, there is a one-to-one correspondence between the set of good $w$-filtrations of $A$ and the set of $w$-valuations of $A$.  

For a non-negative integer $n$ we let
${\mathcal A}_{n}$ denote the class of 
the Poisson algebras $A_{\Omega}$ where 
$\Omega\in \Bbbk[x,y,z]$ is a 
homogeneous element of degree $3+n$
having an isolated singularity.
See Example \ref{xxexa0.1} for the definition
of $A_{\Omega}$ and $P_{\Omega-\xi}$.
Similarly, for $\xi\in \Bbbk$, let 
${\mathcal P}_{n,\xi}$ denote the class 
of the Poisson algebras $P_{\Omega-\xi}$ 
where $\Omega\in \Bbbk[x,y,z]$ is a 
homogeneous element of degree $3+n$
having an isolated singularity. The next statement is basically 
\cite[Lemma 4.8]{HTWZ1} combined with \cite[Lemma 7.6(2)]{HTWZ1}.

Considering $\Bbbk[x,y,z]$ as a $\Bbb{Z}$-graded algebra with the Adams grading, that is, the one in which the degrees of $x$, $y$, and $z$ are all equal to $1$, we have the following.  

For $\xi = 0$, the ideal $(\Omega)$ is a graded ideal, and consequently, $P_{\Omega}$ inherits a $\Bbb{Z}$-graded structure:
\[
P_{\Omega} = \bigoplus_{i \in \Bbb{Z}} P_i.
\]
Note that  $P_i = 0$, for each  $i <0$ and the {\it Adams$^{-Id}$ filtration} of $P_{\Omega}$, denoted by ${\mathbf F}^{-Id}$, is defined by
$$
F^{-Id}_i(P) := \bigoplus_{n \leq -i} P_n
\quad \text{for all $i \in \Bbb{Z}$}.
$$
When $\xi \neq 0$, the Adams grading on $\Bbbk[x,y,z]$ induces a natural filtration on $P_{\Omega - \xi}$, which we denote by $\mathbf{F}^c$. That is, for every $i \in \Bbb{Z}$, we have
\[
F_i^c := \{ f + (\Omega - \xi) \mid f \in \Bbbk[x,y,z] \text{ and } \deg(f) \leq -i \}.
\]

 One can check that  
$\gr_{{\mathbf F}^c} P_{\Omega - \xi} \cong P_{\Omega}$.

\begin{lemma}
\label{xxlem2.14}
Let $B$ be in either ${\mathcal A}_{n}$ or
${\mathcal P}_{n,\xi}$ for some $\xi\in \Bbbk$.
\begin{enumerate}
\item[(1)]
Suppose $w<n$. Then every negatively-${\mathbb N}$-$w$-filtration of $B$ is trivial, or 
equivalently, $u_{good,w}(B)=B$.
\item[(2)]
Suppose that $n>0$ and $B\in {\mathcal P}_{n,\xi}$. Then $u_{good, 0}(B)=B$.
\item[(3)]
Let $B$ be in ${\mathcal P}_{n,\xi}$ where
$\xi\in \Bbbk$. Suppose $w=n>0$. Then there is a 
unique nontrivial 
negatively-${\mathbb N}$-$w$-filtration 
$\mathbf{F}$
of $B$
that is ${\mathbf{F}}={\mathbf{F}}^{-Id}$ if $B\in {\mathcal P}_{n,0}$ or 
${\mathbf{F}}={\mathbf{F}}^{c}$ if $B\in 
{\mathcal P}_{n,\xi}$ if $\xi\neq 0$ as described in \cite[Lemma 4.8(4)]{HTWZ1}. 
As a consequence, $u_{good, n}(B)=\Bbbk$. 
\item[(4)]
Let $n\geq 1$ and 
let $B$ and $B'$ in $\bigcup_{\xi\in \Bbbk}{\mathcal P}_{n,\xi}$. Suppose $\phi\colon
B\longrightarrow B'$ be an injective Poisson algebra 
morphism. Then $\phi$ is a bijection. 
\end{enumerate}
\end{lemma}

\begin{proof}
Assertion (4) is \cite[Lemma 7.6 (2)]{HTWZ1}, where we put $\epsilon = 1$.\\
Assertions (1)–(3) basically follow
from \cite[Lemma 4.8]{HTWZ1}, but let us explain part (1).

Let $K$ be the fraction field of $B$. 
By \cite[Definition 4.1(1)]{HTWZ1},
$$^{w}\Gamma(K):=
\bigcap_{\nu} F^{\nu}_0(K)$$
where $\nu$ runs over all $w$-valuations
on $K$ and $F^{\nu}_0(K) = \{ a \in K \vert v(a) \geq 0\} $. By \cite[Lemma 4.8(1)]{HTWZ1},
$^{w}\Gamma(K)\supseteq B$. 
For every negatively-${\mathbb N}$-$w$-filtration of $B$, say ${\mathbf F}$, 
it defines a $w$-valuation of $B$
by \cite[Lemma 2.6]{HTWZ1}. By \cite[Lemma 2.7]{HTWZ1}, it extends to a $w$-valuation of $K$. Then  equation
$^{w}\Gamma(K)\supseteq B$ implies that 
$F_0(B)=B$. So ${\mathbf F}$ is trivial.
Consequently, $u_{good, w}(B)=B$.
\end{proof}

Next, we establish the relative Dixmier property for certain classes of Poisson algebras. 
Let $n \geq 1$ and let ${\mathcal P}_n$ denote the class of Poisson algebras $\bigcup_{\xi \in \Bbbk} {\mathcal P}_{n,\xi}$. 

Suppose that ${\mathcal G}_{n-1}$ is the class of Adams-connected graded Poisson domains 
$R = \bigoplus_{i \geq 0} R_i$ generated by $R_1$ such that $\{R_1, R_1\} \subseteq R_{n+1}$.  (Note that here $R_0 = \kk$, each $R_i$ is a subspace of $R$, and  $R_i R_j \subseteq R_{i+j}$.)
We now define a negatively-$\mathbb{N}$-filtration on $R$ by setting 
$F_{-i} (R) := \sum_{j \leq i} R_j$. 

Let ${\mathcal F}_{n-1}$ be the class of Poisson firm domains $S$ such that $u_{\mathrm{firm},\,n-1}(S) = \Bbbk$. Note that when  $\Bbbk$ is algebraically closed, every commutative domain is a firm domain.
Therefore, by the  definition, we have ${\mathcal G}_{n-1} \subseteq {\mathcal F}_{n-1}$.

\begin{theorem}
\label{xxthm2.15}
Retain the above notation and assume that $\Bbbk$ is algebraically closed.
Then ${\mathcal P}_{n}$ is ${\mathcal F}_{n-1}$-Dixmier in the sense of Definition {\rm{\ref{xxdef0.4}(1)}}.
\end{theorem}

\begin{proof}
Since $\Bbbk$ is algebraically closed, we have  $u_{good,w}=u_{firm,w}$. Let $A$ and $A'$ be in ${\mathcal P}_{n}$ and let $f\colon A\longrightarrow A'\otimes R$ be an injective Poisson algebra morphism where $R\in {\mathcal F}_{n-1}$. We have:
$$\begin{aligned}
A&=u_{good, n-1}(A) \qquad\; \qquad 
\; \qquad\; \qquad 
{\text{by Lemma \ref{xxlem2.14}(1)}}\\
&\xrightarrow{f} u_{good,n-1}(A'\otimes R) \qquad\;\qquad\;\qquad {\text{by Lemma \ref{xxlem2.11}(1)}}\\
&= u_{good,n-1}(A')\otimes u_{good, n-1}(R)
\qquad {\text{by Lemma \ref{xxlem2.11}(3)}}\\
&=A'\otimes \Bbbk.
\end{aligned}$$
Thus, $f$ can be regarded as an injective map from $A$ to $A' \otimes \Bbbk \cong A'$.
By Lemma \ref{xxlem2.14}(4), the image 
of $f$ is $A'\otimes \Bbbk$. So the assertion follows.
\end{proof}

\subsection{Other applications:
relative cancellation and non-existence of Hopf coactions}
\label{xxsec2.4}
We briefly mention two other applications. The first one concerns relative cancellation
introduced in \cite{HNWZ1}. For more information on the cancellation problem, we refer readers to \cite{GW, GWY} and the survey paper \cite{HTW} for additional related references. 

\begin{definition}
\label{xxdef2.16}
Let ${\mathcal R}$ be a family of Poisson algebras. 
\begin{enumerate}
\item[(1)]
A Poisson algebra $A$ is called 
{\it ${\mathcal R}$-cancellative} if, 
for any Poisson algebra $B$ and any two 
Poisson algebras $R, R'\in {\mathcal R}$ 
of the same Gelfand-Kirillov dimension, a 
Poisson algebra isomorphism $A\otimes R\cong 
B\otimes R'$ implies that $A\cong B$.  
\item[(2)]
A Poisson algebra $A$ is called 
{\it ${\mathcal R}$-bicancellative} if, for 
any Poisson algebra $B$ and any two Poisson 
algebras $R, R' \in {\mathcal R}$ of the same 
Gelfand-Kirillov dimension, a Poisson algebra 
isomorphism $A\otimes R\cong B\otimes R'$ 
implies that $A\cong B$ and $R\cong R'$. 
\item[(3)]
Let ${\mathcal C}$ be another family of 
Poisson algebras. We say the pair $({\mathcal R}, 
{\mathcal C})$ is {\sf bicancellative} if 
for any two Poisson algebras $C,C'\in {\mathcal C}$ 
and any two Poisson algebras $R, R' \in {\mathcal R}$, 
a Poisson algebra isomorphism $C\otimes R\cong 
C'\otimes R'$ implies that $C\cong C'$ and 
$R\cong R'$.  
\end{enumerate}
\end{definition}

\begin{theorem}
\label{xxthm2.17}
Retain the notation as before. 
\begin{enumerate}
\item[(1)]
Every Poisson algebra $A$ in ${\mathcal P}_n$ is 
${\mathcal F}_{n-1}$-cancellative.
\item[(2)]
Every Poisson algebra $A$ in ${\mathcal A}_n$ is 
${\mathcal F}_{n-1}$-cancellative.
\item[(3)]
$({\mathcal F}_{n-1}, {\mathcal P}_{n,0})$ 
is bicancellative.
\item[(4)]
$({\mathcal F}_{n-1}, {\mathcal A}_{n})$ 
is bicancellative.
\end{enumerate}
\end{theorem}

\begin{proof} The proofs of all four parts are 
similar to the proof of \cite[Theorem 2.4(1)]{HNWZ1}, so we only prove part (1). Parts (3) and (4) also follow from \cite[Theorem 2.4(2)]{HNWZ1}.

(1) Let $A$ be in ${\mathcal P}_n$ and $R_1, R_2$
be in ${\mathcal F}_{n-1}$ of the same Gelfand-Kirillov dimension. Suppose $A\otimes R_1
\cong B\otimes R_2$ for another Poisson algebra $B$. 
Since $A\otimes R_1$ is a domain, so is $B$. Applying 
$u_{good, n-1}$ to the isomorphism $A\otimes R_1
\cong B\otimes R_2$ and using Lemma \ref{xxlem2.11} and the fact 
$u_{firm, n-1}(R_1)=u_{firm,n-1}(R_2)=\Bbbk$, we obtain
that 
$$A\cong A\otimes \Bbbk 
\subseteq u_{good,n-1}(A\otimes R_1)
\cong u_{good, n-1}(B\otimes R_2)
\subseteq B\otimes \Bbbk\cong B.$$
Now the proof of \cite[Theorem 2.4(1)]{HNWZ1}
applies. 
\end{proof}

The second application concerns Hopf algebra coactions. We refer the readers to \cite{Ago} for the 
definition of Poisson Hopf algebras, and to \cite{Mon} for the definition of Hopf algebra coaction and 
their related properties.

\begin{theorem}
\label{xxthm2.18}
Let $A$ be in ${\mathcal P}_{n}$. Suppose $H$ is a Poisson Hopf algebra 
that is in ${\mathcal F}_{n-1}$ as a Poisson algebra. Then every Poisson 
$H$-coaction on $A$ is trivial. 
\end{theorem}

\begin{proof}
Suppose $\rho\colon A\to A\otimes H$ is a Poisson Hopf coaction of $H$ on $A$.
By the counit axiom, $\rho$ is an injective map. By the proof of Theorem \ref{xxthm2.15},
the image of $\rho$ is in $A\otimes \Bbbk$ inside $A\otimes H$. This means that the coaction $\rho$ is trivial. 
\end{proof}

\section{Tensor products of Poisson algebras}
\label{xxsec3}
In this section, we assume that $\Bbbk$ is algebraically closed for simplicity. Under this
assumption, the tensor product of two commutative domains is again a domain \cite[Lemma 1.1(1)]{HNWZ1}. Our main goal in
this section is to prove the Dixmier 
property for the tensor product of $P_{\Omega_i-\xi_{i}}$ where $\Omega_i$ are i.s. potentials of distinct degrees
$\geq 5$. We start with some lemmas.

\begin{lemma}
\label{xxlem3.1}
Let $A$ and $B$ be two Poisson domains. Suppose the Poisson center of $Q(A)$ is
$\Bbbk$. Then every nonzero  Poisson ideal of $A\otimes B$ contains an element of the form 
$a\otimes b$ for some $0\neq a\in A$ and $0\neq b\in B$.
\end{lemma}

\begin{proof}
Let $I$ be a nonzero Poisson ideal of $A\otimes B$. 
Let $0\neq c\in I$ and write
$c$ as $\sum_{i=1}^{n} a_i \otimes b_i$ where
$a_i\in A$ and $b_i\in B$. Suppose
$n$ is minimal among all possible 
nonzero $c$ in $I$. We claim that $n=1$.
Suppose to the contrary that $n>1$. 
Then we may assume that $\{a_1,\cdots, a_n\}$
(resp. $\{b_1,\cdots, b_n\}$) 
are linearly independent. 

Case 1: $\{a_1,a_2\}\neq 0$.
Then $I$ contains the element
$\{c, a_1\otimes 1\}=\sum_{i=1}^{n}
\{a_i,a_1\}\otimes b_i=\sum_{i=2}^{n}
\{a_i,a_1\}\otimes b_i$, which is nonzero
since $\{b_1,\cdots, b_n\}$ are 
linearly independent, and $\{a_2,a_1\}\neq 0$. This contradicts the minimality of
$n$.

It remains to consider the following case:

Case 2: $\{a_1,a_2\}=0$. Note that, for every $f \in A$, 
we have  that $ 
(f \otimes 1 ) c = \sum_{i=1}^{n} a_i f \otimes b_i \in I$, which is nonzero,  for any nonzero $f$. By Case 1,
we may further assume that $\{a_1 f, a_2 f\}=0$. So 
$$0=\{a_1f, a_2 f\}
=\{a_1, f\} a_2 f+\{f, a_2\} a_1 f+\{a_1,a_2\}f^2=\{a_1, f\} a_2 f+\{f, a_2\} a_1 f.
$$
Hence $\{a_1,f\}a_2=\{a_2,f\}a_1$. 
This implies that
$$\{a_1 a_2^{-1},f\}=\{a_1,f\} a_2^{-1}-\{a_2, f\} a_2^{-2} a_1=(\{a_1, f\} a_2-\{a_2, f\} a_1) a_2^{-2}=0.$$
Since $f$ is arbitrary, $a_1 a_2^{-1}$ is in the center of $Q(A)$ which is $\Bbbk$ by hypothesis. So $a_{1}$ and $a_{2}$ are linearly dependent, which is a contradiction.
\end{proof}

Let $w$ be a fixed nonnegative integer. Recall that a Poisson $w$-graded algebra 
$P=\bigoplus_{i\ge 0} P_i$ (with connected Adams grading) is called 
{\it projectively simple} if  
\begin{itemize}
\item[(a)] $P_{\ge h}:=\sum_{i\ge h} P_i \neq 0$ for all $h\gg 0$, and
\item[(b)] every nonzero graded Poisson ideal of $P$ contains $P_{\ge h}$ for some $h>0$.
\end{itemize}
Note that since we are using an Adams grading, each $P_i$ is finite-dimensional. Thus condition~(a) is equivalent to $P$ being infinite dimensional, and 
condition~(b) is equivalent to saying that every nonzero graded Poisson ideal of $P$ is cofinite dimensional.  
We refer the reader to \cite{RRZ} for further results and discussion on 
projectively simple rings.

\begin{lemma}
\label{xxlem3.2}
Let $S$ and $A$ be domains, and suppose that $S$ is a Poisson simple algebra. Let $A$ be a
Poisson algebra such that $Q(A)$ has a trivial Poisson center.
\begin{enumerate}
\item[(1)]
If $A$ is Poisson simple, then so is 
$S\otimes A$.
\item[(2)]
Suppose $A$ is a projectively simple connected $w$-graded Poisson algebra. 
Then every graded non-zero Poisson ideal of 
$S\otimes A$ contains $S\otimes A_{\geq h}$ for some $h>0$.
\end{enumerate}
\end{lemma}

In part (2), we use the grading from $A$ and 
define $\deg s=0$ for all $s\in S$. 

\begin{proof}[Proof of Lemma \ref{xxlem3.2}]
Let $I$ be a nonzero Poisson ideal of $S\otimes A$. By Lemma \ref{xxlem3.1}, there is an element $0\neq s\otimes a\in I$ where $s\in S$ and $a\in A$. Consider $J_{a}:=\{s\in S\mid s\otimes a\in I \}$. Using the equality 
$ \{ s'\otimes 1 , s \otimes a\} = 
\{ s', s\} \otimes a 
$, for every $s, s' \in S$,  it is easy to see that 
 $J_{a}$ is a nonzero Poisson ideal of $S$.
Since $S$ is Poisson simple, $J_{a}=S$ and $1\otimes a\in I$.

(1) By the above, we may assume that $1\otimes a\in I$.  Thus  $I\cap (\Bbbk\otimes A)$ is not zero. Since $A$ is Poisson simple, so is $ \Bbbk \otimes A$ and consequently 
$1\otimes 1 \in \Bbbk\otimes A  = I\cap (\Bbbk \otimes A)  \subseteq I$.  So 
$I=S\otimes A$.

(2) Let $I$ be a nonzero graded Poisson ideal of 
$S\otimes A$. 
By the first paragraph, $1\otimes a\in I$
for some $0\neq a\in A$. 
Since $I\cap (\Bbbk\otimes A)$ is not zero and $A$ is projectively simple,
$I\cap (\Bbbk \otimes A)$
contains $1\otimes A_{\geq h}$ for 
some $h>0$. So 
The assertion follows.
\end{proof}

\begin{lemma}
\label{xxlem3.3}
Let $A$ be a Poisson domain such that $u_{\mathrm{good}, w}(A) = A$. 
Let $P$ be a Poisson domain satisfying the following conditions:
\begin{enumerate}
    \item[(i)] $u_{\mathrm{good}, w}(P) = \Bbbk$;
    \item[(ii)] $P$ has a unique  
    nontrivial negatively-${\mathbb N}$-$w$-filtration ${\mathbf G}'$;
    \item[(iii)]
    $\operatorname{gr}_{{\mathbf G}'} P$ is projectively simple domain;
    \item[(iv)] $Q(\operatorname{gr}_{{\mathbf G}'} P)$ has trivial Poisson center.
\end{enumerate}
Then $A \otimes P$ has a unique nontrivial negatively-${\mathbb N}$-$w$-filtration.
\end{lemma}

\begin{proof}
Let ${\mathbf G}'$ be the unique nontrivial   negatively-${\mathbb N}$-$w$-filtration
of $P$ provided by (ii).  
Define a filtration ${\mathbf H}'$ of $A\otimes P$ by $H'_i(A\otimes P): =A\otimes G'_i(P)$ for all $i$. 
Then it is a {{nontrivial}} negatively-${\mathbb N}$-$w$-filtration of $A\otimes P$. It remains to show the uniqueness.

Let ${\mathbf H}$ be any nontrivial  negatively-${\mathbb N}$-$w$-filtration of $A\otimes P$. By restriction, it induces a negatively-${\mathbb N}$-$w$-filtration of $A$, denoted by ${\mathbf F}$ (resp. a negatively-${\mathbb N}$-$w$-filtration  of $P$, denoted by ${\mathbf G})$. Since $u_{good, w}(A)=A$,
${\mathbf F}$ is trivial,  that is  $F_i(A)=A$ for all $i\leq 0$ and $F_i(A)=0$ for all
$i>0$. Since $P$ has the unique {{nontrivial}} negatively-${\mathbb N}$-$w$-filtration,
${\mathbf G}$ is either trivial or equal to ${\mathbf G}'$. If ${\mathbf G}$ is trivial, we obtain that $H_0(A\otimes P)=A\otimes P$ which implies that ${\mathbf H}$ is trivial, yielding a contradiction. So it forces that ${\mathbf G}={\mathbf G}'$. Since 
$A\otimes \Bbbk \subseteq H_0(A\otimes P)$, localizing 
elements in $A\setminus \{0\}$, the filtration ${\mathbf H}$ induces a
negatively-${\mathbb N}$-$w$-filtration of $K\otimes P$ where $K=Q(A)$, which is denoted still by ${\mathbf H}$. Similarly, ${\mathbf H}'$ induces a negatively-${\mathbb N}$-$w$-filtration of $K\otimes P$, which is still denoted by ${\mathbf H}'$. For each $i\leq 0$,
$H'_{i}(K\otimes P)=K\otimes G'_i(P)
\subseteq H_i(K\otimes P)$. By 
\cite[Lemma 2.8(1)]{HTWZ1}, there 
is a natural graded Poisson algebra morphism
$$\phi: \gr_{{\mathbf H}'} K\otimes P
(\cong K\otimes \gr_{{\mathbf G}'} P)
\to \gr_{\mathbf H} K\otimes P.$$ 
{
The kernel of $\phi$ is a graded Poisson ideal, and so by 
Lemma~\ref{xxlem3.2}(2), if $\ker(\phi)$ is nonzero, then 
there exists an integer $m \leq 0$ such that 
$\ker(\phi)$ contains all elements 
$1 \otimes p$ for every 
$p + {\mathbf G}'_{i + 1}(P) \in 
{\mathbf G}'_{i }(P)/ {\mathbf G}'_{i + 1}(P)
$ with $i \leq m$. 
Since $\phi$ maps $1 \otimes p$ to $1 \otimes p$ for all 
$p + {\mathbf G}'_{i + 1}(P) \in 
{\mathbf G}'_{i }(P)/ {\mathbf G}'_{i + 1}(P)
$ and $i < 0$, this implies that 
${\mathbf G}'_{i}(P) = {\mathbf G}'_{i+1}(P)$ for all $i \leq m$, 
which is a contradiction, since 
$\operatorname{gr}_{{\mathbf G}'} P$ is a projectively simple  
and therefore infinite dimensional. 
Thus $\ker(\phi) = 0$ and $\phi$ is injective. 
By \cite[Lemma~2.8(2)]{HTWZ1}, 
${\mathbf H} = {\mathbf H}'$ on $K \otimes P$. 
Restricted to $A \otimes P$, we obtain ${\mathbf H} = {\mathbf H}'$ as required.

}
\end{proof}

\begin{lemma}
\label{xxlem3.4}

Let $\Omega$ be a homogeneous i.s. potential of degree at least $4$. 
Let $K$ be a field which is a Poisson $\kk$-algebra, and let $\xi \in \Bbbk$. 
Then every injective $K$-algebra endomorphism of $K \otimes P_{\Omega-\xi}$ 
that is a Poisson $\kk$-algebra morphism is bijective. 
\end{lemma}
\begin{proof}
Set $w=\deg \Omega -3$. Since ${\rm deg} (\Omega)$ is at least $4$, $w$ is greater than or equal to $1$. By Lemma \ref{xxlem2.14}(3), $P_{\Omega-\xi}$ has a unique {{nontrivial}} negatively-${\mathbb N}$-$w$-filtration such that the associated graded ring is isomorphic to $P_{\Omega}$. Since $\Omega$ has
 an isolated singularity, $P_{\Omega}$ is projectively simple, see \cite  [Lemma 1.6] {HTWZ1}. So $P:= P_{\Omega-\xi}$ satisfies  the hypotheses (i-iv) of Lemma \ref{xxlem3.3} As a result, $K\otimes P_{\Omega-\xi}$ (as a Poisson $\Bbbk$-algebra) has a unique nontrivial negatively-${\mathbb N}$-$w$-filtration, say ${\mathbf F}$, such that the associated graded ring is isomorphic to $K\otimes P_{\Omega}$.
Now let 
\[
\phi : K \otimes P_{\Omega-\xi} \longrightarrow K \otimes P_{\Omega-\xi}
\]
be an injective Poisson {$K$-algebra} endomorphism. Then it is easy to see that 
\[
\phi(\mathbf F) := \{\, F_i + \phi(F_i(K\otimes P_{\Omega-\xi})) \,\}
\]
is a nontrivial negatively–${\mathbb N}$–$w$–filtration of $K\otimes P_{\Omega-\xi}$. This implies that 
\[
\mathbf F = \phi(\mathbf F).
\]
That is, for each $i$,
\[
F_i + \phi(F_i(K\otimes P_{\Omega-\xi})) = F_i,
\]
and consequently
\[
\phi(F_i(K\otimes P_{\Omega-\xi})) \subseteq F_i.
\]
This shows that we obtain a $w$-graded Poisson $\Bbbk$-algebra morphism
\[
\gr \phi : \gr_{\mathbf F}(K\otimes P_{\Omega-\xi}) \longrightarrow 
\gr_{\mathbf F}(K\otimes P_{\Omega-\xi})
\]
sending $x + F_{i+1} \in F_i/F_{i+1}$ to $\phi(x) + F_{i+1}$.

Note that $\gr\phi$ is injective. Indeed, if
\[
G_i := \{\, x \in F_i \mid \phi(x) \in F_{i+1} \,\},
\]
then the collection $\{\, G_i + F_{i+1} \,\}$ is a negatively–${\mathbb N}$–$w$–filtration, which cannot be trivial since $K \otimes P_{\Omega-\xi}$ is not finite-dimensional. By uniqueness of $\mathbf F$, we must have
\[
G_i = F_{i+1},
\]
which shows that $\gr\phi$ is injective.

Recall that each $F_i$ is of the form
\[
F_i = K \otimes G'_i(P_{\Omega-\xi}),
\]
where $G'$ is the unique nontrivial negatively–${\mathbb N}$–$w$–filtration on $P_{\Omega-\xi}$. This implies that each $F_i/F_{i+1}$ is a finite-dimensional $K$-module. Because $\phi$ is a $K$-algebra morphism, for each $i$, and for all $k \in K$ and $x + F_{i+1} \in F_i/F_{i+1}$,
\[
\gr\phi(k(x+F_{i+1})) 
= \gr\phi(kx+F_{i+1}) 
= \phi(kx)+F_{i+1} 
= (k\phi(x))+F_{i+1} 
= k (\,\gr\phi(x+F_{i+1})).
\]
Hence the restriction of $\gr\phi$ to $F_i/F_{i+1}$ is an injective $K$-linear map; since $F_i/F_{i+1}$ is finite-dimensional over $K$, this restriction is also surjective for each $i$.

Suppose that $\phi$ is not surjective. Then we may choose a maximal $i \le 0$ such that there exists
\[
x \in F_i \setminus F_{i+1}
\]
which $x$ is not in the image of $\phi$. Since the restriction of $\gr\phi$ to $F_i/F_{i+1}$ is surjective, we have
\[
F_{i+1} + \phi(F_i) = F_i.
\]
Thus there exist $y \in F_{i+1}$ and $t \in \operatorname{im}(\phi)$ such that
\[
x = y + t.
\]
Now $y$ must lie in $\operatorname{im}(\phi)$; otherwise, we contradict the choice of $i$. But then $x = y + t$ is also in $\operatorname{im}(\phi)$, which is a contradiction.

Therefore, $\phi$ is surjective.
\end{proof}

Note that the above lemma provides a stronger assertion than merely a base change result of \cite[Theorem 0.9]{HTWZ1} since the Poisson structure on $K$ is not trivial. 

\begin{lemma}
\label{xxlem3.5}
Let $\Omega$ be a homogeneous i.s. potential 
of degree $d\geq 5$. Let $K$ be a Poisson 
field. The order of 
$\Aut_{Poi.alg}(K\otimes P_{\Omega-\xi}\mid K)$,
for any $\xi\in \Bbbk$,
is bounded above by $42d(d-3)^2$.
\end{lemma} 

\begin{proof}
We break our proof into three steps. In the first two steps, we consider the case when $\xi=0$ and the last step deals with the case when $\xi\neq 0$.

\noindent
{\bf Step 1:} 
Recall that $P_{\Omega}$ has the standard grading
\[
P_{\Omega} = \bigoplus_{i \in \Bbb Z} P_i,
\]
as explained before Lemma~\ref{xxlem2.14}.
On the other hand,
\[
K \otimes P_{\Omega} = \bigoplus_{i \in \Bbb Z} K \otimes P_i
\]
is also a graded Poisson algebra. The goal of this step is to show that the cardinality of $
\Aut_{\mathrm{gr.Poi.alg}}(K \otimes P_{\Omega} \mid K),$ the subgroup of $\Aut_{\mathrm{Poi.alg}}(K \otimes P_{\Omega} \mid K)$ consisting of grading-preserving automorphisms, is bounded above by
$42\, d(d-3)^2$.

First, we show that this cardinality is bounded by 
the cardinality of $
\Aut_{\mathrm{gr.Poi.alg}}(P'_{\Omega}),$
where $P^{\prime}_{\Omega}:=\Bbbk^{\prime}[x,y,z]/(\Omega)$ and $\Bbbk'$ denotes the Poisson center of $K$.
Note that $\Bbbk \subseteq \Bbbk^{\prime} \subseteq K$, and that $\Bbbk^{\prime}$ is a field with trivial Poisson bracket. 

Now let
\[
\phi \in \Aut_{\mathrm{gr.Poi.alg}}(K \otimes P_{\Omega} \mid K).
\]
Since $
\{\lambda \otimes 1,\; \lambda' \otimes x\} = 0 $,  for any $\lambda \in K$, $\lambda' \in \Bbbk'$, and $x \in P_{\Omega}$, and since $\phi$ fixes $K$ pointwise, it follows that $
\{\lambda \otimes 1,\; \phi(\lambda' \otimes x)\} = 0.$
Writing
\[
\phi(\lambda' \otimes x) = \sum_{i = 1}^{n} \mu_i \otimes x_i,
\]
where $n$ is a finite integer, each $\mu_i \in K$ and the elements $x_i \in P_{\Omega}$ are linearly independent, we conclude that each $\mu_i$ belongs to $\Bbbk'$. Consequently, $
\phi(\lambda' \otimes x) \in \Bbbk' \otimes P_{\Omega}.$

Therefore, restriction induces an injective group homomorphism
\[
\Aut_{\mathrm{gr.Poi.alg}}(K \otimes P_{\Omega} \mid K)
\longrightarrow
\Aut_{\mathrm{gr.Poi.alg}}(\Bbbk' \otimes P_{\Omega} \mid \Bbbk').
\]
Note that
\[
\Bbbk' \otimes P_{\Omega} \cong \Bbbk'[x,y,z]/(\Omega) = P'_{\Omega},
\]
and clearly
\[
\Aut_{\mathrm{gr.Poi.alg}}(\Bbbk' \otimes P_{\Omega} \mid \Bbbk')
=
\Aut_{\mathrm{gr.Poi.alg}}(P'_{\Omega}).
\]

By the final part of the proof of \cite[Theorem~8.2]{HTWZ1}, the order of the group
\[
\Aut_{\mathrm{Poi.alg}}(P'_{\Omega})
\]
is bounded above by $42\, d(d-3)^2$. 
Consequently, the cardinality of
\[
\Aut_{\mathrm{gr.Poi.alg}}(K \otimes P_{\Omega} \mid K)
\]
is also bounded above by $42\, d(d-3)^2$, which proves the claim in this step.\\
\noindent
{\bf Step 2:} We now consider $\Aut_{Poi.alg}(K\otimes P_{\Omega}\mid K)$. Let $\sigma\in \Aut_{Poi.alg}(K\otimes P_{\Omega}\mid K)$. Since
$K\otimes P_{\Omega}$ has the unique 
negatively-${\mathbb N}$-$w$-filtration
${\mathbf F}$, $\sigma$ preserves ${\mathbf F}$. So $\sigma$ induces 
a Poisson automorphism $\sigma':
\gr_{\mathbf F} K\otimes P_{\Omega}
\to \gr_{\mathbf F} K\otimes P_{\Omega}
$ which preserves the degree. 
  Bearing in mind that $ 
\gr_{\mathbf F} K\otimes P_{\Omega} = 
K\otimes P_{\Omega}
$, we have that 
$\sigma'$
is in $\Aut_{gr.Poi.alg}(K\otimes P_{\Omega}\mid K)$. Replacing $\sigma$ by 
$\sigma (\sigma')^{-1}$ we may assume that
$\sigma'$ is the identity. We claim that 
$\sigma$ is the identity in this case.
Since $\sigma'$ is the identity, $\sigma(x)=x+a$, $\sigma(y)=y+b$ and $\sigma(z)=z+c$ where $a,b,c\in K$.  (Note that in $K \otimes P_{\Omega}$ we identify
$K$ with $K \otimes \kk$, and $\kk \otimes P_{\Omega}$ with $P_{\Omega}$.)
Observe that  in $K\otimes \Bbbk[x,y,z]$, $$\Omega(x+a, y+b,z+c)
=\Omega(x,y,z)+a \Omega_{x} +b \Omega_{y} +c \Omega_z + ldt$$
where $ltd$ is a sum of terms of Adams degree 
less or equal to $d-2$. Hence in $K\otimes P_{\Omega}$, 
$$\begin{aligned}
0&=\sigma(0)=\sigma(\Omega)
=\Omega (x+a, y+b,z+c)\\
&=\Omega(x,y,z)+a \Omega_{x} +b \Omega_{y} +c \Omega_z + ldt
=a \Omega_{x} +b \Omega_{y} +c \Omega_z + ldt
\end{aligned}$$
which implies that $a \Omega_{x} +b \Omega_{y} +c \Omega_z=0$ when restricted to degree $d-1$. Since $\Omega$ has i.s., 
$\Omega_x, \Omega_y, \Omega_z$ are $\Bbbk$-linearly independent. Consider these as elements in $K\otimes P_{\Omega}$, these are $K$-linearly independent, which implies that $a=b=c=0$ as required. Therefore, every Poisson algebra automorphism in $\Aut_{Poi.alg}(K\otimes P_{\Omega}\mid K)$ preserves the grading.
Thus, the assertion  for 
$\xi =  0$ follows from Step 1.

\noindent
{\bf Step 3:} Let $\xi\neq 0$. We will 
construct a group homomorphism from
$\Aut_{Poi.alg}(K\otimes P_{\Omega-\xi}\mid K)\to \Aut_{Poi.alg}(K\otimes P_{\Omega}\mid K)$. Let $\sigma\in \Aut_{Poi.alg}(K\otimes P_{\Omega-\xi}\mid K)$. Since
$K\otimes P_{\Omega -\xi}$ has the unique 
negatively-${\mathbb N}$-$w$-filtration
${\mathbf F}$ [Lemma \ref{xxlem3.3}], $\sigma$ preserves ${\mathbf F}$. So $\sigma$ induces 
a Poisson automorphism $\sigma':
\gr_{\mathbf F} K\otimes P_{\Omega-\xi}
\to \gr_{\mathbf F} K\otimes P_{\Omega-\xi}
$ which preserves the degree.  Bearing in mind that $ 
\gr_{\mathbf F} K\otimes P_{\Omega-\xi } = 
K\otimes P_{\Omega}
$, we have that 
$\sigma'$
is in $\Aut_{gr.Poi.alg}(K\otimes P_{\Omega}\mid K)$. 

Using the fact that ${\mathbf F}$ is unique, one can easily show that $\sigma\to \sigma'$ is a group homomorphism. We claim that this is an 
injective map. Suppose $\sigma'$ is the 
identity, by using the filtration ${\mathbf F}$, we have 
$\sigma(x)=x+a$, $\sigma(y)=y+b$ and $\sigma(z)=z+c$ where $a,b,c\in K$.
Note that in $K\otimes \Bbbk[x,y,z]$, $$\Omega(x+a, y+b,z+c)-\xi
=\Omega(x,y,z)-\xi +a \Omega_{x} +b \Omega_{y} +c \Omega_z + ldt$$
where $ltd$ is a sum of terms of degree 
less than or equal to  $d-2$. Hence in $K\otimes P_{\Omega}$, 
$$\begin{aligned}
0&=\sigma(0)=\sigma(\Omega-\xi)
=\Omega (x+a, y+b,z+c)\\
&=\Omega(x,y,z)-\xi+a \Omega_{x} +b \Omega_{y} +c \Omega_z + ldt
=a \Omega_{x} +b \Omega_{y} +c \Omega_z + ldt
\end{aligned}$$
which implies that $a \Omega_{x} +b \Omega_{y} +c \Omega_z=0$ in the associated graded ring $\gr_{\mathbf F}
(K\otimes P_{\Omega-\xi})$. Since $\Omega$ has i.s., 
$\Omega_x, \Omega_y, \Omega_z$ are $\Bbbk$-linearly independent. Consider these as elements in $K\otimes P_{\Omega}$, these are $K$-linearly independent, which implies that $a=b=c=0$. So $\sigma$ is the 
identity and the map $\sigma\to \sigma'$ is injective. The assertion follows from Step 1.
\end{proof}

We are now ready to prove the main results.

\begin{proof}[Proof of Theorem \ref{xxthm0.2}]
(1) We use induction on $d$. If $d=1$, the assertion follows from 
\cite[Theorem 0.9]{HTWZ1}. 
Now we assume $d\geq 2$. Without loss of generality, we may assume that $\deg \Omega_1> \deg \Omega_2> \cdots >\deg \Omega_d=:m$.
Let $A$ be the Poisson domain $\bigotimes_{i=1}^{d-1} P_{\Omega_i-\xi_i}$. Then $P=A\otimes P_{\Omega_d-\xi_d}$. By Lemma \ref{xxlem2.14}(1),  $u_{good, m-3}(P_{\Omega_i-\xi_i})=P_{\Omega_i-\xi_i}$ for all $i\leq d-1$. By  Lemma \ref{xxlem2.11} (3), 
$u_{good, m-3}(A)=A$. By Lemma \ref{xxlem2.14}(3), $u_{good,m - 3}(P_{\Omega_d-\xi_d})=\Bbbk$. 
Let $\phi$ be an injective Poisson algebra endomorphism of $P=A\otimes P_{\Omega_d-\xi_d}$. Applying $u_{good, m -3}$ to $\phi$ and using Lemma \ref{xxlem2.14}(2), we obtain an injective Poisson algebra morphism
$$f\colon =u_{good, m -3}(\phi)\colon A\otimes \Bbbk \longrightarrow A\otimes \Bbbk.$$
By the induction hypothesis, $f$ is a bijective morphism. Replacing $\phi$ by $(f^{-1}\otimes Id) \phi$, we may assume that $f$ is the identity, namely, the restriction of $\phi$ to
$A$ is the identity. Localizing elements in $A\setminus \{0\}$, we obtain an injective Poisson algebra endomorphism
$$\psi\colon K\otimes P_{\Omega_{d}-\xi_d} 
\longrightarrow K\otimes P_{\Omega_d-\xi_d}$$
such that the restriction of $\psi$ on $A\otimes P_{\Omega_d-\xi_d}$ is $\phi$.
Since 
$f$ is identity,  $\psi$ is $K$-morphism and now by 
Lemma \ref{xxlem3.4}, $\psi$ is bijective. 
Since $\psi$ is the identity on $K$, by Lemma \ref{xxlem3.5}, there is an integer $m$
such that $\psi^{m}$ is the identity. This implies that $\phi^{m}$ is the identity. Therefore, $\phi$ is a bijection as required.

(2) This follows from the next corollary.

(3) This follows from induction, Lemma 
\ref{xxlem3.2} and \cite[Theorem 1 and Corollary 2]{UZ}. 
\end{proof}

\begin{corollary}
\label{xxcor3.6}
Let $\Omega_1,\cdots,\Omega_d$ be a set of i.s.
potentials of distinct Adams degrees with 
$\deg \Omega_i\geq 5$ for all $i$. Let 
$\xi_1,\cdots,\xi_d$ be a set of scalars in 
$\Bbbk$. Let $P$ be the Poisson algebra 
$P_{\Omega_1-\xi_1}\otimes \cdots \otimes 
P_{\Omega_d-\xi_d}$.
\begin{enumerate}
\item[(1)]
Let $C$ be a Poisson domain. Then 
$\Aut_{Poi.alg}(C\otimes P\mid C)$ is bounded by $\prod_{i=1}^{d} \{ 42 \deg \Omega_i (\deg \Omega_i-3)^2\}$.
\item[(2)]
Let $L_{\Lambda}(\Bbbk)$ be a simple Poisson torus as defined in Example {\rm{\ref{xxexa0.5}}}. Then
there is a short exact sequence
$$1\to \Aut_{Poi.alg}(L_{\Lambda}(\Bbbk) \otimes P\mid L_{\Lambda}(\Bbbk) ) \to \Aut_{Poi.alg}(L_{\Lambda}(\Bbbk) \otimes P)
\to \Aut_{Poi.alg}(P)\to 1. $$
\end{enumerate}
\end{corollary}

\begin{proof}
(1) Without loss of generality, we may assume $C$ is a Poisson field. In particular, by \cite[Corollary~1.4]{HNWZ1}, we may assume that $C$ is $u_{\mathrm{good},\,w}$-maximal for all $w \geq 0$.
 We use induction on $d$. If $d=0$, 
there is nothing to prove. If 
 $d=1$, the statement follows from Lemma \ref{xxlem3.5}.
Now we assume $d\geq 2$. Similar to the 
proof of Theorem \ref{xxthm0.2} and 
assume that $\deg \Omega_1> \deg \Omega_2> \cdots >\deg \Omega_d=\colon m$. Let $A$ be the Poisson domain $ C \otimes  \bigotimes_{i=1}^{d-1} P_{\Omega_i-\xi_i}$. Then $C\otimes P = A\otimes P_{\Omega_d-\xi_d}$.
So $A$ is 
$u_{good, m- 3}$-maximal and $P_{\Omega_d-\xi_d}$ is 
$u_{good, m - 3}$-minimal. By Theorem \ref{xxthm2.12}(2),
there is a short exact sequence of groups
{\small $$1\longrightarrow \Aut_{Poi.alg}(A\otimes P_{\Omega_d-\xi_d}\mid A) \longrightarrow \Aut(A\otimes P_{\Omega_d-\xi_d}\mid C)\longrightarrow \Aut_{Poi.alg}(A\mid C)\to 1.$$
}
Let $K$ be the fractional field of $A$, then 
$\Aut_{Poi.alg}(A\otimes P_{\Omega_d-\xi_d}\mid A)$ is a subgroup of $\Aut_{Poi.alg}(K\otimes P_{\Omega_d-\xi_d}\mid K)$ where latter is bounded by $ 42 \deg \Omega_d (\deg \Omega_d-3)^2$ by 
Lemma \ref{xxlem3.5}. Now the assertion follows from the induction and the above exact sequence. 

(2) Note that $L_{\Lambda}(\Bbbk)$ is $u_{good}$-maximal (see \cite[Corollary~1.4]{HNWZ1}) and $P$ is $u_{firm}$-minimal (follows from 
Theorem \ref{xxthm2.15}). The assertion now follows from Theorem \ref{xxthm2.4}(2), where we take $C = \kk$.
\end{proof}

\begin{proof}[Proof of Theorem \ref{xxthm0.6}]
Suppose that  $A$ is  the Poisson algebra $L_{\Lambda}(\Bbbk) \otimes P$ where $P=\bigotimes _{i=1}^{d} P_{\Omega_i-\xi_i}$. First, we prove that $A$ is Dixmier.
We use induction on $d$. If $d=0$, the 
assertion follows from Theorem \ref{xxthm2.6}, with this remark that $\kk$  is algebraically closed and every commutative domain is a firm domain.
Now we assume $d\geq 1$. Without loss of generality, 
we may assume that $\deg \Omega_1> \deg \Omega_2> 
\cdots >\deg \Omega_d=\colon m$. Let $A'$ be the 
Poisson domain $L_{\Lambda}(\Bbbk) \otimes 
\prod_{i=1}^{d-1} P_{\Omega_i-\xi_i}$. Then $A 
=A'\otimes P_{\Omega_d-\xi_d}$. By Lemma 
\ref{xxlem2.14}(1), $u_{good, m - 3}(P_{\Omega_i-
\xi_i})=P_{\Omega_i-\xi_i}$ for all $i\leq d-1$. Therefore using Lemma \ref{xxlem2.11} (2) and \cite[Corollary~1.4]{HNWZ1}, $u_{good, m- 3}(A')=A'$. We also use Lemma~\ref{xxlem2.14}(3) to obtain $u_{\mathrm{good},\,m-3}(P_{\Omega_d-\xi_d})=\Bbbk$.

Let $\phi$ be an injective Poisson algebra endomrophism of $A = A' \otimes P_{\Omega_d-\xi_d}$. Applying $u_{good, m-3}$ to $\phi$ and using Lemma \ref{xxlem2.11}(2), we obtain an injective Poisson algebra morphism
$$f\colon =u_{good, m - 3}(\phi)\colon A'\otimes \Bbbk \longrightarrow A'\otimes \Bbbk.$$
By the induction hypothesis, $f$ is a bijective morphism. Replacing $\phi$ by $(f^{-1}\otimes Id) \phi$, we may assume that $f$ is the identity, namely, the restriction of $\phi$ to
$A'$ is the identity. Inverting all elements in $A'\setminus \{0\}$, we obtain an injective Poisson algebra endomorphism
$$\psi: K\otimes P_{\Omega_{d}-\xi_d} 
\longrightarrow K\otimes P_{\Omega_d-\xi_d}$$
such that the restriction of $\psi$ on $A'\otimes P_{\Omega_d-\xi_d}$ is $\phi$.
Since $\psi$ is the identity on $K$, by Lemma \ref{xxlem3.4}, $\psi$ is bijective and consequently 
 by 
Lemma \ref{xxlem3.5}, there is an integer $m$
such that $\psi^{m}$ is the identity. This implies that $\phi^{m}$ is the identity. Therefore, $\phi$ is a bijection as required.

Now, to conclude Theorem~\ref{xxthm0.6}, it is enough to show that $A$ and $P$ are  ${\mathcal F}_{w}$-Dixmier whenever
\[
w \leq \min\{\deg \Omega_i\} - 1.
\]
Indeed, as defined earlier,  ${\mathcal F}_{w}$ is the class of Poisson firm domains $S$ such that
\[
u_{\mathrm{firm},\,w}(S) = \Bbbk.
\]
Since $\Bbbk$ is algebraically closed, every commutative domain is a firm domain. Therefore, by definition, we have $
{\mathcal G}_{w} \subseteq {\mathcal F}_{w}.$ 
So let us see  that $A$ is ${\mathcal F}_{w}$-Dixmier if $w\leq \min\{\deg \Omega_i\}-1 = m - 4$. 
By the above proof, one sees that $A$ is 
$u_{good, w}$-maximal. Then every $R$ in 
${\mathcal F}_{w}$ is $u_{good, w}$-minimal.
The assertion follows from Theorem \ref{xxthm2.12}(3). So this takes care of the case when $A=L_{\Lambda}(\Bbbk) \otimes 
\bigotimes_{i=1}^{d} P_{\Omega_i-\xi_i}$.

When $P= \bigotimes_{i=1}^{d} P_{\Omega_i-\xi_i}$,
by Theorem \ref{xxthm0.2}, $P$ is Dixmier.
Note that $P$ is $u_{good, w}$-maximal and 
that $R\in {\mathcal F}_{w}$ is 
$u_{good, w}$-minimal. By Theorem \ref{xxthm2.12}(3), $P$ is $R$-Dixmier.
\end{proof}

\section{Integral closure and 
complementary space}
\label{xxsec4}
In this section, we introduce the concepts of complementary spaces and modules, which help us better understand the Dixmier and $R$-Dixmier properties.  Among other results, we provide the proof of Theorem~\ref{xxthm0.9}.  

\subsection{Closure of $A$}
\label{xxsec4.1}
First, we recall some lemmas from commutative algebra.

\begin{lemma}
\label{xxlem4.1}
Let $\{A_{i}\}$ be a set of integrally closed subalgebras inside a field $K$. Then the intersection $A := \bigcap_{i} A_i$ is integrally closed.
\end{lemma}

\begin{proof}
Let $x\in Q(A)\setminus A$ that is integral 
over $A$.  Then $x$ is a root of a monic polynomial $f(t)$ with coefficients in $A$.
So, for each $i$, 
$x\in Q(A_i)$ is a root of the same 
monic polynomial $f(t)$ whose coefficients 
are in $A\subseteq A_i$. Since each $A_i$ is 
integrally closed, $x\in A_i$ for all $i$.
So $x\in A$ as required.
\end{proof}

We say a Poisson domain $A$ is {\it integrally closed} if it is integrally closed in its fractional field $Q(A)$  as a commutative algebra. Let $\mathcal{K}$ denote the category of all Poisson fields over $\Bbbk$. 
For every $K\in \mathcal{K}$, let us define 
\[
\Phi(K)=\{\text{all Poisson subalgebras of}\, K\, \text{that are integrally closed}\}.
\]
First of all, we need to define a closure operation, which furnishes every Poisson domain $A$ with its closure in its fractional field $Q(A)$. 
\begin{definition}
\label{xxdef4.2}
Let $A$ be a Poisson domain.
\begin{enumerate}
\item[(1)]
Let $K\in {\mathcal K}$ and $f\colon A\longrightarrow K$ be an injective Poisson morphism, 
whose extension to $Q(A)$ is also denoted by $f$. The {\it closure} of $A$ over $f$ 
is defined to be 
$$\moc_{f}(A)
=\bigcap_{f(A)\subseteq B\in 
\Phi(K)} B.$$
Since $f(Q(A))\in \Phi(K)$, we have 
$A \subseteq f^{-1}(\moc_{f}(A))\subseteq Q(A)$. 
\item[(2)]
The {\it closure} of $A$ is 
defined to be 
$$\moc(A)
=\bigcap_{{f\colon A\hookrightarrow K}, K\in {\mathcal K}}
f^{-1}(\moc_{f}(A)).$$
\end{enumerate}
\end{definition}

The following is our main definition.

{
\begin{definition}
\label{xxdef4.3}
Let $A$ be a Poisson domain and $\moc(A)\subseteq Q(A)$ its
closure.

\begin{enumerate}
\item[(1)]
A subspace $V\subseteq \moc(A)$ is called a 
\it{complementary space} of $A$ if $A\oplus V = \moc(A)$.
Any such space is denoted by $\CompS(A)$.

\item[(2)]
The \it{complement module} of $A$ is
\[
\CompM(A) := \moc(A)/A,
\]
viewed as an $A$-module.
\end{enumerate}
\end{definition}

It is clear that $\CompS(A)\cong \CompM(A)$ as vector spaces. Note that $\CompM(A)$ is uniquely
defined, but $\CompS(A)$ is not. We will explore some applications of complementary spaces.

}

\begin{lemma}
\label{xxlem4.4}
Let $A$ be a Poisson domain. Then $\moc(A)=
f^{-1}(\moc_{f}(A))$ for every injective Poisson algebra map $f: A\to K$ where $K$ is a Poisson
field.
\end{lemma}
\begin{proof}
Take an injective Poisson algebra map $f \colon A \to K$, where $K$ is a Poisson field.
By definition, $\moc_{f}(A)$ is the intersection of a family of integrally closed Poisson domains.
Since the intersection of Poisson subalgebras is again a Poisson algebra, it follows from Lemma~\ref{xxlem4.1}
that $\moc_{f}(A)$ is an integrally closed Poisson domain.
Therefore, using Part~(1) of Definition~\ref{xxdef4.2}, we have
$f^{-1}(\moc_{f}(A)) \in \Phi(Q(A))$; moreover, $f^{-1}(\moc_{f}(A))$ contains $A$.

Similarly, if $g \colon A \to K'$ is another injective Poisson algebra homomorphism, where $K'$ is a Poisson field,
then $g^{-1}(\moc_{g}(A)) \in \Phi(Q(A))$, and $g^{-1}(\moc_{g}(A))$ contains $A$.
We have
\[
f(A) \subseteq f\bigl(g^{-1}(\moc_{g}(A))\bigr) \subseteq f(Q(A)) \subseteq K.
\]
Since $g^{-1}(\moc_{g}(A))$ is a Poisson algebra, so is $f\bigl(g^{-1}(\moc_{g}(A))\bigr)$.
Consequently, by the definition of $\moc_{f}(A)$, we obtain
\[
\moc_{f}(A) \subseteq f\bigl(g^{-1}(\moc_{g}(A))\bigr).
\]
This implies that
\[
f^{-1}(\moc_{f}(A)) \subseteq g^{-1}(\moc_{g}(A)).
\]
Finally, using the definition of $\moc(A)$ in Part~(2) of Definition~\ref{xxdef4.2} and the fact that
$g \colon A \to K'$ was an arbitrary injective Poisson algebra homomorphism into a Poisson field, we conclude that
\[
\moc(A) = f^{-1}(\moc_{f}(A)).
\]

\end{proof}

\begin{lemma}
\label{xxlem4.5}
Let $A$ be a Poisson domain and $B$ be a Poisson subalgebra of $Q(A)$ containing $A$. If $B$ is integrally closed 
and is integral over $A$, then $\moc(A)=B$.
\end{lemma}

\begin{proof}
By Lemma \ref{xxlem4.4}, $\moc(A)=\moc_{f}(A)$ where $f$ is the inclusion $A\to Q(A)$. 
Since $B$ is integrally closed, $B\in \Phi(Q(A))$. Consequently, $\moc(A)\subseteq B$. Now let $C$ be any element in $\Phi(Q(A))$ containing $A$. Since $B$ is integral 
over $A$, $B\subseteq C$. So $B$ is the 
minimal element in $\Phi( Q(A))$  containing $A$. Therefore
$\moc(A)=B$.
\end{proof}

\subsection{Complementary  spaces and complement modules}
\label{xxsec4.2}

\begin{lemma}
\label{xxlem4.6}
Let $A$ be a Poisson domain.
\begin{enumerate}
\item[(1)]
$c(A)$ is an algebraic invariant.
\item[(2)]
$\CompM(A)$ (resp. $\CompS(A)$) is an algebraic
invariant.
\item[(3)]
There are inclusions of groups
$$\Aut(A)\subseteq \Aut(\moc(A))\subseteq \Aut(Q(A)).$$
\end{enumerate}
\end{lemma}
\begin{proof}
(1) Let $\phi\colon A\longrightarrow A'$ be a Poisson algebra 
isomorphism. A Poisson algebra map $f\colon A\longrightarrow K$,
where $K$ is a Poisson field, is  injective if 
and only if $f'\colon =f\circ \phi^{-1}\colon A'\longrightarrow K$ is 
 injective.
By definition, $\moc_{f}(A)=\moc_{f'}(A')$. Hence,
$\phi$ induces an isomorphism from $f^{-1}(\moc_{f}(A))
\to (f')^{-1}(\moc_{f'}(A))$. By Lemma \ref{xxlem4.4},
$\phi$ induces an isomorphism from $\moc(A)\to 
\moc(A')$.

(2) The assertion follows from the following 
diagram
$$\begin{CD}
0@>>> A @>>> \moc(A) @>>> \CompM(A) @>>> 0\\
@. @V \stackrel{\phi} {\cong} VV @VV \stackrel{\phi} {\cong} V @VVV @.\\
0@>>> A' @>>> \moc(A') @>>> \CompM(A') @>>> 0.\\
\end{CD}$$

(3) The assertions follow from part (1).
\end{proof}

\begin{lemma}
\label{xxlem4.7}
Let $A$ be a Poisson subalgebra of $B$ that is a 
Poisson domain and  $\partial$ be a weak 
dimension function. Then, there is a natural inclusion of Poisson algebras
$\moc(A)\longrightarrow \moc(B)$ such that the following 
diagram commute
\begin{equation}
\label{E4.7.1}\tag{E4.7.1}
\begin{CD}
0@>>> A @>>> \moc(A) @>>> \CompM(A) @>>> 0\\
@. @VVV @VVV @VVV @.\\
0@>>> B @>>> \moc(B) @>>> \CompM(B) @>>> 0\\
\end{CD}
\end{equation}
where the map $\CompM(A)\longrightarrow \CompM(B)$ is induced by the left commutative square.
\end{lemma}

\begin{proof}
We have an inclusion map
$f\colon A\hookrightarrow B\hookrightarrow Q(B)=: K$. Let $f'$ be the inclusion
of $B$ into  $K$. Then both $\moc_{f}(A)$ 
and $\moc_{f'}(B)$ are Poisson subalgebras of
$K$ by the definition. Also by the definition, $\moc_{f}(A)$
is a Poisson subalgebra of $\moc_{f'}(B)$ both 
of which are considered as Poisson 
subalgebras of $K$. By Lemma \ref{xxlem4.4},
$\moc(A)=f^{-1}(\moc_{f}(A))=\moc_{f}(A)$,  since
$f$ is the inclusion. Similarly, $\moc(B)=
(f')^{-1}(\moc_{f'}(B))=\moc_{f'}(B)$. Therefore, we obtain the left commuting 
square of diagram \eqref{E4.7.1}. By the definition,
the top and bottom rows are short exact sequences.
Therefore, we have an induced map 
$\CompM(A)\longrightarrow \CompM(B)$ that makes the right square 
commutative.

\end{proof}

\subsection{Dixmier property vs Complementary spaces}
\label{xxsec4.3}
Let $n \geq 0$ be a nonnegative integer. We say that a Poisson algebra $A$ has the {\it  $n$-Dixmier property} if $A$ is $\kk[x_1,\ldots,x_n]$-Dixmier. Clearly, the $0$-Dixmier property is just the Dixmier property. 
An {\it  augmentation} of a $\kk$-algebra $A$ is a $\kk$-algebra homomorphism
\[
\pi \colon A \longrightarrow \kk
\]
whose restriction to $\kk \subseteq A$ is the identity.
The kernel of $\pi$, denoted by $\ker(\pi)$, is called the {\it  augmentation ideal}.
If $A$ is a Poisson algebra, we additionally require that $\pi$ be a Poisson homomorphism.

We say that $A$ is {\it  augmented $R$-Dixmier} if, for any Poisson morphism
\[
\phi \colon A \longrightarrow A \otimes R
\]
and any augmentation $\pi \colon R \to \kk$ such that the composition
\[
A \xrightarrow{\ \phi\ } A \otimes R \xrightarrow{\ \mathrm{Id}_A \otimes \pi\ } A\otimes \Bbbk \cong A
\]
is the identity map, it follows that
\[
\operatorname{im}(\phi)= A \otimes \kk.
\]

\begin{lemma}
\label{xxlem4.8}
Let $B$ be a Poisson domain that is integrally 
closed. Let $A$ be a Poisson subalgebra of $B$
such that {$B\subseteq Q(A)$ with} $\dim_{\Bbbk} B/A<\infty$.
\begin{enumerate}
\item[(1)]
Suppose $B$ has the Dixmier property. Then so does
$A$.
\item[(2)]
Let $R$ be a Poisson domain such that $B\otimes R$
is an integrally closed domain. If $B$ satisfies $R$-Dixmier property,
so does $A$.
\item[(3)]
Let $n$ be a positive integer. Suppose $B$ has 
$n$-Dixmier property. Then so does $A$.
\item[(4)]
Suppose that $B$ is an augmented $R$-Dixmier and that $B\otimes R$ is integrally closed. Then $A$ 
is augmented $R$-Dixmier.
\end{enumerate}
\end{lemma}

\begin{proof}
(1) Let $\sigma: A\longrightarrow A$ be an injective Poisson 
endomorphism of $A$. By Lemma \ref{xxlem4.7}(1),
we have a commutative diagram
\begin{equation}
\notag
\begin{CD}
0@>>> A @>>> B=\moc(A) @>>> \CompM(A) @>>> 0\\
@. @V\sigma VV @V\sigma' VV @VVg V @.\\
0@>>> A @>>> B=\moc(A) @>f >> \CompM(A) @>>> 0\\
\end{CD}
\end{equation}
where $\sigma'$ denotes the induced Poisson
endomorphism of $B$. Since $\sigma'$ and $f$ are surjective, $g$ is surjective. By hypothesis, $\CompM(A)$ is finite-dimensional.
So $g$ is bijective. By the $5$-lemma, $\sigma$
is an automorphism as required. Therefore, Theorem \ref{xxthm0.9} holds. 

(2) Let $f\colon A\longrightarrow A\otimes R$ be an 
injective Poisson algebra morphism. 
By Lemma \ref{xxlem4.7}(1), there is
an induced Poisson algebra morphism
$f'\colon \moc(A)\longrightarrow \moc(A\otimes R)$. Since both $B$ and $B\otimes R$ are integrally closed with $A\subseteq B\subseteq Q(A)$ and $A\otimes R\subseteq B\otimes R\subseteq Q(A\otimes R)$, 
we have that $B=\moc(A)$ and $B\otimes R=\moc(A\otimes R)$  by Lemma \ref{xxlem4.5}. Thus $f'$ is 
the induced injective Poisson algebra morphism $f'\colon \moc(A)=B \longrightarrow \moc(A\otimes R)=B\otimes R$.
Since $B$ is $R$-Dixmier, we have ${\rm{im}}(f')=
B\otimes \Bbbk$. This implies that 
${\rm{im}}(f)\subseteq (A\otimes R)
\cap (B\otimes \Bbbk)=
A\otimes \Bbbk$. By part (1), $A$ is
Dixmier. Hence 
${\rm{im}}(f)=A\otimes \Bbbk$.

(3) This is a special case of part (2),  indeed  $B \otimes \kk [x_1, \cdots, x_n] \cong  B [x_1, \cdots, x_n[$ is an integrally closed domain. 

(4) Let $\pi: R\to \Bbbk$ be an augmentation of 
$R$. Suppose $A\xrightarrow{\phi} A\otimes R\xrightarrow{Id_A\otimes \pi} A$ is the identity.
Using Lemma \ref{xxlem4.7} and the fact that 
$\moc (A) = B$ and $\moc (A\otimes R) = B \otimes R$, the map $\phi$ induces a Poisson algebra morphism
$\sigma: B\to B\otimes R$. So we have a commutative
diagram
$$\begin{CD}
A@>\phi >> A\otimes R @>>> A\\
@VVV @VVV @VVV @.\\
B@>\sigma>> B\otimes R @>>> B.
\end{CD}$$
Since the composition of the top row is the 
identity, the composition of the bottom row is also
the identity. Since $B$ is augmented $R$-Dixmier,
${\rm{im}}(\sigma)=B\otimes \Bbbk$. Restricted to $A$,
we have ${\rm{im}}(\phi)\subseteq A\otimes \Bbbk$. 
Since $A$ is Dixmier by part (1), we obtain ${\rm{im}}(\phi)= A\otimes \Bbbk$ as required.
\end{proof}

\subsection{Field extension}

\label{xxsec4.4}
It is natural to ask about the Dixmier property under field extensions, and in this short subsection, we address this question with the following proposition.

\begin{proposition}
\label{xxprop4.9}
Let $A$ and $R$ be Poisson domains over the base field $\Bbbk$. Let $B=A\otimes K$ 
for some commutative field $K$
with a trivial Poisson bracket.
So $B$ can be viewed as a Poisson
algebra over $K$.
\begin{enumerate}
\item[(1)]
If $B$ is Dixmier as a Poisson $K$-algebra, then $A$ is Dixmier as a Poisson $\Bbbk$-algebra.
\item[(2)]
If $B$ is $R\otimes K$-Dixmier as a Poisson $K$-algebra,
then $A$ is $R$-Dixmier as a Poisson $\Bbbk$-algebra.
\item[(3)]
Suppose $\Aut_{\Bbbk}(A)=\Aut_{K}(B)$ and $B$ is Dixmier as a Poisson $K$-algebra. Then $A$ is $K$-Dixmier 
as a Poisson $\Bbbk$-algebra. 
\end{enumerate}
\end{proposition}

\begin{proof}
(1) Let $\phi\colon A\longrightarrow A$ be an injective Poisson $\Bbbk$-algebra endomorphism and let $\sigma=\phi\otimes K$. Then 
$\sigma$ is an injective Poisson $K$-algebra endomorphism of 
$B$. Since $B$ is Dixmier as a Poisson $K$-algebra, $\sigma$ is surjective. This forces $\phi$ to be surjective.

(2) Let $\phi\colon A\longrightarrow A\otimes R$ be an injective Poisson $\Bbbk$-algebra morphism and 
let $\sigma=\phi\otimes K$. We consider $\sigma$ as a Poisson $K$-algebra morphism from $B$ to  
$(A\otimes R)\otimes K=(A\otimes K) \otimes_{K} (R\otimes K)=B\otimes_{K} S$,  where $S=R\otimes K$. 
Since $B$ is $S$-Dixmier as a Poisson $K$-algebra, we have that ${\rm{im}}(\sigma)=B=A\otimes K$. 
This forces that ${\rm{im}}(\phi)=A\otimes \Bbbk$ inside $A\otimes R$.
Thus $A$ is 
$R$-Dixmier as a Poisson $\Bbbk$-algebra.

(3) By part (1), $A$ is Dixmier as a Poisson $\Bbbk$-algebra. Let 
\(\phi \colon A \longrightarrow A \otimes K\) be an injective Poisson 
$\Bbbk$-algebra morphism. Define $\sigma$ to be the composition
\[
B := A \otimes K
\xrightarrow{\phi \otimes \mathrm{Id}_K} (A \otimes K) \otimes K
\cong A \otimes (K \otimes K)
\xrightarrow{\mathrm{Id}_A \otimes m_K}
A \otimes K =: B.
\]
Let us show that $\sigma$ is injective. Since  the canonical associativity isomorphism
\[
\iota\colon (A \otimes K)\otimes K \xrightarrow{\ \cong\ } A \otimes (K \otimes K)
\]
is bijective, so the composition $\iota \circ (\phi \otimes \mathrm{Id}_K)$ is injective.

Next, consider the multiplication map \(m_K \colon K \otimes K \to K\). It admits a $\Bbbk$-linear section
\[
s\colon K \longrightarrow K \otimes_{\Bbbk} K,\qquad s(\lambda)=\lambda \otimes 1,
\]
since \(m_K(\lambda \otimes 1)=\lambda\). Tensoring with $A$ yields a $\Bbbk$-linear map
\[
\mathrm{Id}_A \otimes s \colon A \otimes K \longrightarrow A \otimes (K \otimes K),
\]
which, for every \(m\) in the image of \(\iota \circ (\phi \otimes \mathrm{Id}_K)\), satisfies
\[
(\mathrm{Id}_A \otimes s)\circ (\mathrm{Id}_A \otimes m_K)(m)=m.
\]
Thus $\sigma$ is injective.
Since $B$ is Dixmier as a Poisson $K$-algebra, $\sigma$ is an automorphism. 
By the hypothesis, $\Aut_{\Bbbk}(A)=\Aut_{K}(B)$, there is an 
$\pi\in \Aut_{\Bbbk}(A)$ such that $\sigma=\pi\otimes K$.
For every $a\in A$,
$$\pi(a)\otimes 1=\sigma(a\otimes 1)=
(Id_A\otimes m_K)\circ (\phi\otimes Id_K)(a\otimes 1)=
(Id_A\otimes m_K)(\phi(a)\otimes 1)=\phi(a)$$
which implies that ${\rm{im}}(\phi)\subseteq A\otimes \Bbbk$.
Since $A$ is Dixmier as a Poisson $\Bbbk$-algebra, 
${\rm{im}}(\phi)= A\otimes \Bbbk$ as required.
\end{proof}

\subsection{Poisson algebras with trivial automorphism groups}
\label{xxsec4.5} In this subsection, using some results from this section, we introduce certain Poisson algebras with trivial Poisson automorphism groups. Let $A$ be a graded commutative domain that is a Poisson algebra. For example $A=\Bbbk[x_1,\cdots,x_n]$
as in Example \ref{xxexa1.7}. Next, we will define 
an infinite  class of Poisson algebras that are between $A$ and
$A(d)$ for some $d\geq 4$.

\begin{example}
\label{xxexa4.10}
Let $d\geq 4$ and let $\zeta\in A$ 
which is of the form $\zeta=\sum_{i=2}^{d-1} \zeta_i$ where $\zeta_i$ is  either zero or homogeneous of degree $i$. Suppose $\zeta^2\in A(d)$. Then 
$A(d,\zeta):=A(d)+\Bbbk \zeta$ is Poisson subalgebra of $A$ containing $A(d)$. 
\end{example}

\begin{theorem}
\label{xxthm4.11}
Let $A(d,\zeta)$ be defined as in Example \ref{xxexa4.10}. Suppose that 
\begin{enumerate}
\item[(a)]
Every automorphism of $A$ is graded.
\item[(b)]
For every automorphism $Id_{A}\neq \sigma\in \Aut_{Poi.alg}(A)$,
$\sigma(\zeta)\not\in \Bbbk \zeta$.
\end{enumerate}
Then $\Aut_{Poi.alg}(A(d,\zeta))$ is trivial. 
\end{theorem}

\begin{proof}
Let $\sigma$ be an automorphism of $A(d,\zeta)$ that is not the identity.
By Lemma \ref{xxlem4.6}(3), $\sigma$ extends to
an automorphism of $A$. Since $\sigma$ is graded and $\sigma(\zeta)\not\in \Bbbk \zeta$, it cannot be an automorphism of $A(d,\zeta)$, contradiction.
The assertion follows.
\end{proof}

There are many examples of such $A$ and $\zeta\in A$
such that $\Aut_{Poi.alg}(A(d,\zeta))$ is trivial. Basically, a generic $\zeta$ should work. Below is an explicit
example. 

\begin{example}
\label{xxexa4.12}
Let $\Omega=x^{n}+y^n+z^{n}$ for some 
$n\geq 5$. Let $A=A_{\Omega}$ as defined in 
Example \ref{xxexa0.1}. By \cite[Theorem 8.2]{HTWZ1}, every automorphism of $A$ 
is graded and so condition (a) of 
Theorem \ref{xxthm4.11} holds. Now let $d\geq 8$ and let $\zeta=x^{d-1}+y^{d-1}+y^{d-2}+z^{d-2}+z^{d-3}+x^{d-3}
+x^{d-4}+y^{d-4}$. 
Indeed, $\Aut_{Poi.alg} (A)$ is described explicitly in \cite[Theorem 8.2 and Proposition 8.4]{HTWZ1}. As a result, 
one can easily check that $\sigma(\zeta)\not\in \Bbbk \zeta$. By Theorem \ref{xxthm4.11}, $\Aut_{Poi.alg}(A(d,\zeta))$ is trivial.
\end{example}

\subsection*{Acknowledgments}
The authors would like to thank the organizers of the workshop “Poisson Geometry and Artin–Schelter Regular Algebras-Hangzhou”, during which this joint work was initiated. Nazemian was supported by the Austrian Science Fund (FWF), grant P 36742; Y.-H. Wang was supported by the NSF of China (No. 11971289);
and Zhang was partially supported by the US National Science Foundation (No. DMS-2302087).


\begin{thebibliography}{10}

\bibitem{AvdE}
P.K. Adjamagbo and A. van den Essen, 
A proof of the equivalence of the Dixmier, Jacobian and Poisson conjectures,
{\it Acta Math. Vietnam.} {\bf 32} (2007), no. 2-3, 205--214. 

\bibitem{Ago}
A.L. Agore, Universal coacting Poisson Hopf algebras,
{\it  Manuscripta Math.} 
{\bf 165} (2021), no. 1-2, 255--268.

\bibitem{AW}
V. Artamonov and R. Wisbauer, Homological properties of quantum
polynomials, {\it  Algebr. Represent. Theory} {\bf 4(3)} (2001), 219-247.



\bibitem{BCW} 
H. Bass, E. Connell, and D. Wright, 
The Jacobian conjecture : Reduction of degree and formal
expansion of the inverse, 
{\it  Bull. Amer. Math. Soc.}, Vol. {\bf 7}, Num. 2, (1982), 
pp 287--330.


\bibitem{BLSM} 
J. Bell, S. Launois, O. S\`{a}nchez, and R.  Moosa, 
Poisson algebras via model theory and differential-algebraic geometry, 
{\it  J. Eur. Math. Soc. (JEMS)} {\bf 19} (2017), no. 7, 2019--2049.



\bibitem{BK}
A. Belov-Kanel and M. Kontsevich,
The Jacobian conjecture is stably equivalent to the Dixmier conjecture,
{\it   Mosc. Math. J.} {\bf 7} (2007), no. 2, 209--218.



\bibitem{CO} E.-H. Cho and S.-Q. Oh, 
Semiclassical limits of Ore extensions and a Poisson generalized Weyl algebra, 
{\it  Lett. Math. Phys.} {\bf 106} (2016), no. 7, 997--1009.



\bibitem{Dix} 
J. Dixmier, 
Sur les alg\'{e}bres de Weyl, 
{\it  Bull. Soc. Math. France} {\bf 96} (1968) 209--242.

\bibitem{GW} 
J. Gaddis and X.-T. Wang, 
The Zariski cancellation problem for Poisson algebras, {\it  J. Lond. Math. Soc. (2)} {\bf 101} (2020), no. 3, 1250--1279. 

\bibitem{GWY} J. Gaddis, X.-T. Wang, and D. Yee, 
Cancellation and skew cancellation for Poisson algebras, 
{\it  Math. Z.} {\bf 301} (4) (2022) 3503--3523.

\bibitem{Goo1} K.R. Goodearl, 
A Dixmier-Moeglin equivalence for Poisson algebras with torus actions, In Algebra and its applications,
{\it  Contemp. Math.} {\bf Vol. 419},  pages 131--154, Amer. Math. Soc., Providence, RI, 2006.

\bibitem{Goo2} K.R. Goodearl, 
Semiclassical limits of quantized coordinate rings, 
{\it  Advances in ring theory}, 165--204, Trends Math., Birkhäuser/Springer Basel AG, Basel, 2010.


\bibitem{GL} K.R. Goodearl and S. Launois, The Dixmier-Moeglin equivalence and a Gel'fand-Kirillov problem for Poisson polynomial algebras, {\it  Bull. Soc. Math. France} {\bf 139} (2011), no. 1, 1--39.

\bibitem{HNWZ1}
H.-D. Huang, Z. Nazemian, Y.-H. Wang, and J.J. Zhang, 
Relative cancellation, to appear in {\it Proc. Amer. Math. Soc.}, arXiv:2503.10083.

\bibitem{HNWZ2}
H.-D. Huang, Z. Nazemian, Y.-H. Wang, and J.J. Zhang, 
Universal homogeneity, in progress. 


\bibitem{HTW}
H.-D. Huang, X. Tang, and X.-T. Wang, 
A survey on Zariski cancellation problems for noncommutative and Poisson algebras,
{\it  Contemp. Math.} {\bf Vol. 801} (2024), Amer. Math. Soc., [Providence], RI, pp. 125--141.

\bibitem{HTWZ1}
H.-D. Huang, X. Tang, X.-T. Wang, and J.J. Zhang,
Poisson valuations, 
{\it  J. Algebra} {\bf 683} (2025), 1--59.

\bibitem{HTWZ2}
H.-D. Huang, X. Tang, X.-T. Wang, and J.J. Zhang,
Valuation method for Nambu-Poisson algebras, 
preprint (2023), arXiv:2312.00958. 


\bibitem{JO} 
D.A. Jordan and S.-Q. Oh, 
Poisson spectra in polynomial algebras, 
{\it  J. Algebra} {\bf 400} (2014), 56--71.



\bibitem{KL1}
A.P. Kitchin and S. Launois, 
Endomorphisms of quantum generalized Weyl
algebras, 
{\it  Lett. Math. Phys.} {\bf 104} (2014), no. 7, 837--848.

\bibitem{KL2} 
A.P. Kitchin and S. Launois, 
On the automorphisms of quantum Weyl algebras, 
{\it  J. Pure Appl. Algebra} {\bf 223} (2019), no. 4, 1514--1530.

\bibitem{LS} 
S.  Launois and O. S\`{a}nchez, 
On the Dixmier-Moeglin equivalence for Poisson-Hopf algebras, 
 {\it Adv. Math.} {\bf 346} (2019), 48--69.

\bibitem{Possionstructers}
C.~Laurent-Gengoux, A.~Pichereau and P.~Vanhaecke,
\textit{Poisson Structures},
Universitext, Springer, 2013.


\bibitem{LuWW1}
J. Luo, S.-Q. Wang, and Q.-S. Wu, 
Twisted Poincar\'{e} duality between Poisson homology and Poisson cohomology, 
{\it  J. Algebra} {\bf 442} (2015) 484--505

\bibitem{LuWW2} 
J. Luo, X.-T. Wang, and Q.-S. Wu, 
Poisson Dixmier-Moeglin equivalence from a topological point of view, 
{\it  Israel J. Math.} {\bf 243} (2021), no. 1, 103--139.

\bibitem{LvWZ} 
J.-F. Lü, X.-T. Wang, and G.-B. Zhuang, 
Universal enveloping algebras of Poisson Ore extensions, 
 {\it Proc. Amer. Math. Soc.} {\bf 143} (2015), no. 11, 4633--4645.


\bibitem{Mon} 
S. Montgomery, 
{\it  Hopf Algebras and their Actions on Rings},
CBMS Regional Conference Series in Mathematics, {\bf 82}. 
Published for the Conference Board of the Mathematical 
Sciences, Washington, DC; by the American Mathematical 
Society, Providence, RI, 1993. 

\bibitem{RRZ}
Z. Reichstein, D. Rogalski, and J.J. Zhang,
{\it  Projectively simple rings},
Adv.\ Math.\ \textbf{203} (2006), no.~2, 365--407.

\bibitem{Ric}
L. Richard, 
Sur les endomorphismes des tores quantiques, 
{\it  Comm. Algebra} {\bf 30} (11) (2002) 5283--5306.

\bibitem{RS}
L. Rowen and D.J. Saltman, 
Tensor products of division algebras and fields, {\it  J. Algebra} {\bf 394} (2013), 296--309.

\bibitem{Tan1}
X. Tang, 
Algebra endomorphisms and derivations of some localized down-up algebras, 
{\it  J. Algebra Appl.} {\bf 14} (2015), no. 3, 1550034, 14 pp.

\bibitem{Tan2}
X. Tang, 
Automorphisms for some symmetric multiparameter quantized Weyl algebras and their localizations,
{\it  Algebra Colloq.} {\bf 24} (2017), no. 3, 419--438.

\bibitem{Tan3}
X. Tang, 
The automorphism groups for a family of generalized Weyl algebras, 
{\it  J. Algebra Appl.} {\bf 17} (2018), no. 8, 1850142, 20 pp.

\bibitem{Tsu1}
Y. Tsuchimoto, 
Preliminaries on Dixmier conjecture, 
{\it  Mem. Fac. Sci. Kochi Univ.}, Ser. A
Math. {\bf 24} (2003), 43--59.

\bibitem{Tsu2}
Y. Tsuchimoto,
Endomorphisms of Weyl algebras and $p$--curvature,
{\it  Osaka J. Math.} {\bf 42(2)} (2005), 435--452.

\bibitem{UZ}
U. Umirbaev and V. Zhelyabin, A Dixmier theorem for Poisson enveloping algebras, J. Algebra {\bf 568} (2021), 576--600.

\bibitem{Van} M. Vancliff, Primitive and Poisson spectra of twists of polynomial rings, 
{\it  Algebr. Represent. Theory} {\bf 2} (1999), no. 3, 269--285.

\end{thebibliography}
\end{document}